\documentclass{article}

\usepackage{amssymb,amsthm,amsmath}

\title{Permanental Point Processes on Real Tori, Theta Functions and Monge-Amp\`ere Equations}
\author{Jakob Hultgren}

\newtheorem{theorem}{Theorem}
\newtheorem{lemma}{Lemma}
\newtheorem{proposition}{Proposition}
\newtheorem{corollary}{Corollary}
\newtheorem{definition}{Definition}\theoremstyle{definition}
\newtheorem{remark}{Remark}\theoremstyle{remark}
\theoremstyle{definition}

\providecommand{\norm}[1]{\lVert#1\rVert}
\providecommand{\bracket}[2]{\left\langle #1,#2\right\rangle}

\newcommand{\Z}{\mathbb{Z}}
\newcommand{\R}{\mathbb{R}}

\newcommand{\C}{\mathbb{C}}
\newcommand{\Q}{\mathbb{Q}}
\newcommand{\N}{\mathbb{N}}

\newcommand{\M}{\mathcal{M}}

\renewcommand{\P}{\mathcal{P}}

\newcommand{\dom}{\operatorname{dom}}

\newcommand\MA{\operatorname{MA}}
\newcommand\Ric{\operatorname{Ric}}

\newcommand\perm{\operatorname{perm}}

\begin{document}
\maketitle
\begin{abstract}
Inspired by constructions in complex geometry we introduce a thermodynamic framework for Monge-Amp\`ere equations on real tori. We show convergence in law of the associated point processes and explain connections to complex Monge-Amp\`ere equations and optimal transport. 
\end{abstract}
%\renewcommand{\abstractname}{R\'esum\'e}
%\begin{abstract}
%On introduit un cadre thermodynamique pour des \'equations Monge-Amp\`eres sur des tores r\'eelles, qui est inspir\'e par des constructions en geometrie complexe. On d\'emontre la convergence en loi pour les processus al\'eatoires correspondantes, expliquant les liens avec des \'equations Monge-Amp\`eres complexes et le transport optimal.
%\end{abstract}

\setcounter{tocdepth}{1}
\tableofcontents

\section{Introduction}\label{SectIntro}
In a series of papers culminating in \cite{BermanBiRat} Berman introduce a thermodynamic framework for complex Monge-Amp\`ere equations. In particular, he shows how the Monge-Amp\`ere measures of solutions to complex Monge-Amp\`ere equations can be seen as limits of canonically defined ($\beta$-deformations of) determinantal point processes. Inspired by this we will introduce a thermodynamic framework for real Monge-Amp\`ere equations on the real torus $X=\R^n/\Z^n$. Using certain families of functions analogous to theta functions on Abelian varieties we construct permanental point processes on $X$. Our first result is that, as long as the Monge-Amp\`ere equation admits a unique solution, the point processes defined by the statistic mechanical framework converges in law towards the Monge-Amp\`ere measure of this unique solution. Equivalently, and in the language of thermodynamics, under absence of first order phase transitions the microscopic setting admits a macroscopic limit that is determined by the Monge-Amp\`ere equation.  

The real torus should be seen as one of several settings where strong connections between complex geometry, real Monge-Amp\`ere equations and optimal transport are manifested (the related case of toric manifolds is treated in \cite{Berman}). We will exploit these connections to produce semi-explicit approximations of optimal transport maps on $X$ (see Corollary~\ref{CorrOptTrans}). As such, this work ties in with the seminal works by McCann \cite{McCann} and Cordero-Erasquin \cite{Cordero} on optimal transport on Riemannian manifolds.  

Moreover, motivated by the difficult problem of singular K\"ahler-Einstein metrics of (almost everywhere) positive curvature on complex varieties we propose a corresponding real Monge-Amp\`ere equation on $X$ (see equation \eqref{MAEqSpec} below). The assumption of no first order phase transition always holds for positive temperature. However, a reflection of the fact that the related complex geometric problem is one of \emph{positive} curvature is that the statistical mechanic setting for \eqref{MAEqSpec} is of \emph{negative} temperature. As a second result, by proving a uniqueness theorem for Monge-Amp\`ere equations of independent interest (see Theorem~\ref{ThmFUniq}), we rule out first order phase transitions down to the critical temperature of $-1$. In a future paper we hope to address the question of uniqueness for temperatures smaller than $-1$, which might be seen as the analog of the problem studied in \cite{LinWang}.

%%%%%%%%%%%%%%%%%%%%%%%%%%%%%%%%%%%%%%%%%%%%%%%%%%%%%%

\subsection{Setup}
Let $dx$ be the standard volume measure on $X$ induced from $\R^n$. Let $\beta$ be a real constant and $\mu_0$ a probability measure on $X$, absolutely continuous and with smooth, strictly positive density with respect to $dx$. Given the data $(\mu_0,\beta)$ we will consider the real Monge-Amp\`ere equation on $X$ given by
\begin{equation} \MA(\phi) = e^{\beta \phi}d\mu_0. \label{MAEqGen}\end{equation}
Here $\MA$ is the Monge-Amp\`ere operator defined by
\begin{equation} \phi\mapsto \det(\phi_{ij}+\delta_{ij}) dx. \label{MAOp} \end{equation}
where $(\phi_{ij})$ is the Hessian of $\phi$ with respect to the coordinates on $X$ induced from $\R^n$ and 
$\delta_{ij}$ is the Kronecker delta. As usual we demand of a solution $\phi:X\rightarrow \R$ that it is twice differentiable and quasi-convex in the sense that $(\phi_{ij}+\delta_{ij})$ is a positive definite matrix.

We will pay specific attention to the case when $\mu_0$ is chosen as the measure
$$ \gamma = \sum_{m\in \Z^n} e^{-|x-m|^2/2}dx. $$
We get the equation
\begin{equation} \MA(\phi) = e^{\beta \phi}\gamma. \label{MAEqSpec} \end{equation}
As mentioned above this equation has an interpretation in terms of complex geometry. For $\beta=-1$, \eqref{MAEqSpec} arises as the "push forward" of a twisted K\"ahler-Einstein equation on the Abelian variety $\C^n/4\pi\Z^n+i\Z^n$. A more detailed exposition of this relation will follow in Section \ref{SectionEqPush}. 
%For now we will just point out that compared to \cite{Berman}, an important feature of this new setting is the sign of $\beta$. The negative sign of $\beta$ reflects the fact that the solutions of the corresponding complex equation determines metrics of (almost everywhere) \emph{positive} curvature. From the point of view of complex geometry this is, compared to \emph{negative} curvature, the more subtle case. Further, in a thermodynamic interpretation of the point processes, negative $\beta$ corresponds to so called negative temperature. 

%%%%%%%%%%%%%%%%%%%%%%%%%%%%%%%%%%%%%%%%%%%%%%%%%%%%%%

\subsection{Construction of the Point Processes}\label{SectPointProc}
The point processes we will study arise as the so called "$\beta$-deformations" of certain \emph{permanental point processes} (see \cite{HoughEtAl} for a survey). Let's first recall the general setup of a permanental point process with $N$ particles. We begin by fixing a set of $N$ \emph{wave functions} on $X$
$$ S^{(N)} = \{\Psi^{(N)}_1, \ldots, \Psi^{(N)}_N\}. $$
This defines a matrix valued function on $X^N$ 
$$ (x_1,\ldots,x_N)\rightarrow (\Psi_i(x_j)). $$
Recall that the permanent of a matrix $A=a_{ij}$ is the quantity
$$ \sum_{\sigma}\prod_i a_{i, \sigma(i)} $$
where the sum is taken over all permutations of the set $\{1\ldots, N\}$. Together with the background measure $\mu_0$ this defines a symmetric probability measure on $X^N$
\begin{equation} \perm(\Psi^{(N)}_i(x_j)) d\mu_0^{\otimes N}/Z_N, \label{EqPerm} \end{equation}
where $Z_N$ is a constant ensuring the total mass is one. This is a \emph{pure} permanental point process. We will define, for each $k\in \N$, a set of $N=N_k$ wave functions and, for a given $\beta\in \R$, study the so called $\beta$-deformations of \eqref{EqPerm} 
\begin{equation} \mu_{\beta}^{(N)} = \left(\perm(\Psi_i(x_j)\right)^{\beta/k} d\mu_0^{\otimes N}/Z_{\beta,N} \label{EqBetaPerm} \end{equation}
where, as above, $Z_{\beta,N}$ is a constant ensuring the total mass is one. We will now define the sets of wave functions. Note that $\mu_\beta^{(N)}$ does not depend on the order of the element in $S^{(N)}$. For each positive integer $k$, let
$$ S^{(N)} = \{\Psi_p^{(N)}: p\in \frac{1}{k}\Z^n/\Z^n \} $$
where 
$$ \Psi_p^{(N)}(x) = \sum_{m\in \Z^n+p} e^{-k|x-m|^2/2}dx. $$

Before we move on we should make a comment on the notation. We get $N=N_k=k^n$. Throughout the text, in formulas where both $N$ and $k$ occur, the relation $N=k^n$ will always be assumed. 

Finally, we will make two remarks on the definitions. In \cite{NegeleOrland} permanental point processes are used to model a bosonian many particle system in quantum mechanics. In that interpretation $\Psi_i^{(N)}$ defines a 1-particle wave function and the permanent above is the corresponding $N$-particle wave function defined by $\Psi_{1}^{(N)},\ldots, \Psi_{N}^{(N)}$. Secondly, we will explain in Section \ref{SectionDetPerm} how the wave functions arises as the "push forward" of $\theta$-functions on $\C^n/(4\pi\Z^n+i\Z^n)$.

%%%%%%%%%%%%%%%%%%%%%%%%%%%%%%%%%%%%%%%%%%%%%%%%%%%

\subsection{Main Results}
Denote the space of probability measures on $X$ by $\mathcal{M}_1(X)$ and consider the map $\delta^{(N)}:X^N\rightarrow \mathcal{M}_1(X)$
$$ \delta^{(N)}(x) = \delta^{(N)}(x_1,\ldots, x_N) = \frac{1}{N}\sum_{i=1}^N \delta_{x_i}. $$
Let $x=(x_1,\ldots, x_N)\in X^N$ be the random variable with law $\mu_{\beta}^{(N)}$. Its image under $\delta^{(N)}$, $\delta^{(N)}(x)$, is the \emph{empirical measure}. This is a random measure with law given by the push-forward measure
\begin{equation} \Gamma_\beta^{(N)} = \left(\delta^{(N)}\right)_* \mu^{(N)}_ {\beta} \in \mathcal{M}_1(\mathcal{M}_1(X)) \label{PFMeasures} \end{equation}
Our results concern the weak* limit of $\Gamma_\beta^{(N)}$ as $N\rightarrow \infty$. In particular we will show, in some cases, that the limit is a dirac measure concentrated at a certain $\mu_*\in \M_1(X)$ related to \eqref{MAEqGen} or \eqref{MAEqSpec}. Loosely speaking, this means $\mu_*$ can be approximated by sampling larger and larger point sets on $X$ according to $\mu_\beta^{(N)}$. 

\begin{theorem}\label{MainThmGen}
Let $\mu_0\in\M_1(X)$ be absolutely continuous and have smooth, strictly positive density with respect to $dx$. Let $\Gamma^{(N)}$ be defined as above and let $\beta\in \R$.  
Assume also that \eqref{MAEqGen} admits a unique solution, $\phi_*$.
Then  
\begin{equation} \Gamma^{(N)}_\beta \rightarrow \delta_{\mu_*} \label{EqMainThmGen}\end{equation}
in the weak* topology of $\M_1(\M_1(X))$, where $\mu_* = MA(\phi_*)$. 
\end{theorem}
\begin{remark}\label{RemBetaPos}
The assumption that \eqref{MAEqGen} admits a unique solutions is always satisfied when $\beta>0$. This follows from standard arguments (see Theorem \ref{ThmUniqPos}). However, the case $\beta<0$ is a lot more subtle. In our second result we show that, in the special case $\mu_0=\gamma$, the assumption holds for certain negative values of $\beta$ as well.
\end{remark}
\begin{theorem}\label{ThmFUniq}
Assume $\mu_0=\gamma$ and $\beta\in [-1,0)$. Then equation \eqref{MAEqSpec} admits a unique solution.
\end{theorem}

Note that if $\beta\not=0$ and $\mu_* = MA(\phi_*)dx$ where $\phi_*$ is a solution to \eqref{MAEqGen}, then $\phi_*$ can be recovered from $\mu_*$ as 
$\phi_* = \frac{1}{\beta}\log \rho$
where $\rho$ is the density of $\mu_*$ with respect $\mu_0$. 
In fact we get the following corollary of Theorem \ref{MainThmGen}.
\begin{corollary}\label{CorrGen}
Let $\mu_0\in \M_1(X)$ be absolutely continuous and have smooth, strictly positive density with respect to $dx$. Let $\beta\not=0$. Assume also that \eqref{MAEqGen} admits a unique solution, $\phi_*$. Let $\phi_N:X\rightarrow \R$ be the function defined by
$$ \phi_N(x_1) = \frac{1}{\beta}\log \int_{X^{N-1}}\left(\perm(\Psi^{(N)}_{p_i}(x_j)\right)^{\beta/k} d\mu^{\otimes (N-1)}(x_2,\ldots,x_n)/Z_{\beta,N}. $$
Then $\phi_N$ converges uniformly to $\phi^*$.
\end{corollary}

If we put $\beta=0$ in \eqref{MAEqGen} we get the inhomogenous Monge-Amp\`ere equation. Solutions then determine Optimal Transport maps on $X$. Now, although Corollary \ref{CorrGen} doesn't cover the case $\beta=0$, by considering $\mu_{\beta_N}^{(N)}$ for the sequence of constants $\beta_N=1/N$ we will be able to produce explicit approximations of optimal transport maps. However, when working with optimal transport it is natural to consider a more general setting than the one proposed for equation \eqref{MAEqGen}. Because of this we will not state this corollary here but postpone it to Section \ref{SectApprox}. 

%%%%%%%%%%%%%%%%%%%%%%%%%%%%%%%%%%%%%%%%%%%%%%%%%%%%%%

\subsection{Outline}\label{SectOutline}
\paragraph{Convergence in Theorem \ref{MainThmGen} and a Large Deviation Principle}
Theorem~\ref{MainThmGen} will follow from a \emph{large deviation principle} for the sequence $\Gamma^{(N)}$ (see Theorem~\ref{MainThmLDP}). This large deviation principle provides a quantitative description of the convergence in Theorem~\ref{MainThmGen}, recording the speed of convergence in a \emph{rate function} $G:\M_1(X) \rightarrow [0,\infty)$, satisfying $\inf G = 0$ and a \emph{rate} $\{r_N\}\subset \R$ such that $r_N\rightarrow \infty$ as $N\rightarrow \infty$. We will give a formal definition of large deviation principles in Section \ref{SectionLDP}. Roughly speaking, a large deviation principle with rate function $G$ and rate $r_N$ holds if, for $U\subset \M_1(X)$, the probability $\Gamma(U)$ behaves as 
$$ e^{-r_N \inf_U G} $$
as $N\rightarrow \infty$. 
This means $\Gamma^{(N)}$, for large $N$, is concentrated where $G$ is small. In particular, if $G$ admits a unique minimizer, $\mu_*$, (where $G=0$) then it follows that $\Gamma^{(N)}$ converges in the weak* topology to $\delta_{\mu_*}$. 

\paragraph{Proof of the Large Deviation Principle}
It turns out that the rate function above is related to the Wasserstein metric of optimal transport. In Section \ref{SectOT} we will recall some basic facts about optimal transport. In particular, we explain how Kantorovich' duality principle gives an explicit formula for the Legendre transform of the squared Wasserstein distance from a fixed measure. The proof of Theorem \ref{MainThmLDP} is given in Section \ref{SectionLDP} and it is divided into two parts of which the first part uses this explicit formula. In the first part, given in Section \ref{SectBetaInf}, we take a sequence of constants $\beta_N$ such that $\beta_N\rightarrow \infty$ and study the family $\{\Gamma_{\beta_N}^{(N)}\}$. In the thermodynamic interpretation this means we are studying the zero temperature limit of the system. Using the formula given by Kantorovich duality and the G\"artner-Ellis theorem, relating the moment generating functions of $\Gamma^{(N)}_{\beta_N}$ to the Legendre transform of a rate function, we prove a large deviation principle for this family (see Theorem \ref{ThmBetaInf}). In the second part of the proof we show how the large deviation principle in Theorem \ref{MainThmLDP} can be deduced from this. This is based on essentially well known arguments. However, for completeness we give a proof of this in Section \ref{SectGenCase}. It turns out that the crucial point is the equicontinuity and uniform boundedness of the (normalized) energy functions 
$$ -\frac{1}{kN}\log\perm(\Psi^{(N)}_i(x_j)). $$
These properties will follow from equicontinuity properties and bounds on the wave functions $\Psi^{(N)}_i$ and we give a proof of these properties in Section \ref{SectProofLDP}. 

\paragraph{Connection to the Monge-Amp\`ere Equation}
The final ingredients in the proof of Theorem \ref{MainThmGen} are given in Section \ref{SectLemmaF} and Section \ref{SectDual} (essentially by Lemma \ref{LemmaF} and Theorem \ref{LemmaDual}). These sections connect the large deviation principle above with the Monge-Amp\`ere equation \eqref{MAEqGen}. Note that, as $\inf G=0$, $G$ admits a unique point where $G=0$ if and only if $G$ admits a unique minimizer. We apply a variational approach to \eqref{MAEqGen}. Uniqueness and existence of solutions is studied by means of a certain energy functional on $C(X)$ whose stationary points corresponds to weak solutions of \eqref{MAEqGen}. The rate function above, $G$, is closely related to this energy functional. This relation encodes the fact that minimizers of $G$ arise as the Monge-Amp\`ere measures of solutions to \eqref{MAEqGen}. Moreover, it follows from this relation that $G$ admits a unique minimizer if the energy functional does, which is true if and only if \eqref{MAEqGen} admits a unique solution. 

\paragraph{Theorem \ref{ThmFUniq}}
Existence of weak solutions will follow from the variational approach and compactness properties of the space of quasi convex functions on $X$ (see Section \ref{SectExistence}) and regularity will follow from results by Cafarelli explained in Lemma \ref{LemmaRegularity}. These type of existence results for Monge-Amp\`ere equations on affine manifolds was originally proven by Caffarelli and Viaclovsky \cite{CaffarelliViaclovsky} on the one hand and Cheng and Yau \cite{ChengYau} on the other. However, we will provide an alternative proof based on the variational principle above. Uniqueness, which is the main new contribution in this chapter is proved in Section \ref{SectStrictConv}. Here we look at the space of quasi-convex functions equipped with an affine structure different from the standard one. It will then follow from the Prekopa inequality that the energy functional associated to \eqref{MAEqSpec} is strictly convex with respect to this affine structure, hence admits no more than one minimizer. This is an extension of an argument used in \cite{BermanBerndtsson} to prove uniqueness of K\"ahler-Einstein metrics on toric Fano manifolds. Curiously, there doesn't seem to be any direct argument for this using the Prekopa theorem on Riemannian manifolds (see \cite{CorderoEtAl}). Instead, we need to lift the problem to the covering space $\R^n$ and use that $\gamma$ is the push forward of a measure on $\R^n$ with strong log-concavity properties. 

\paragraph{Geometric Motivation}
In Section \ref{SectGeo} we explain the connections to the point processes on compact K\"ahler manifolds introduced by Berman in \cite{BermanBiRat}. More precisely, we explain the connection with a complex Monge-Amp\`ere equations on $\C^n/4\pi\Z^n+i\Z^n$ and how the wave functions and permanental point processes defined here are connected to theta-functions and determinantal point processes on $\C^n/4\pi\Z^n+i\Z^n$. Finally, in Section \ref{SectApprox} we show how the connection to optimal transport can be used to get explicit approximations of optimal transport maps on $X$. 

We end this section with a comment. While some parts of Section~\ref{SectionLDP} might be well known to readers with a probabilistic background and, likewise, some parts of Section~\ref{SectRate} might be familiar to readers with a background in geometry or optimal transport we nevertheless want to provide a paper that is accessible to readers from all three of these fields. This should (at least partly) explain the length of the paper.

%%%%%%%%%%%%%%%%%%%%%%%%%%%%%%%%%%%%%%%%%%%%%%%%%%

%%%%%%%%%%%%%%%%%%%%%%%%%%%%%%%%%%%%%%%%%%%%%%%%%%%

%%%%%%%%%%% Section 2 %%%%%%%%%%%%%%%%%%%%%%%%%%%%%%%%%%

\section{Preliminaries: Optimal Transport on Real Tori}\label{SectOT}
In this section we will recall some basic theory of optimal transport. The content of the chapter is well known. Early contributors to the theory are Cordera-Erasquin \cite{Cordero} who established a a theory of optimal transport on real tori and McCann \cite{McCann} who took it to the very general setting of Riemannian manifolds. The reason for this is the close relation between optimal transport and real Monge-Amp\`ere equations. The most important part is Corollary \ref{CorrDXi}. There we explain how Kantorovich' duality theorem give a variational approach to real Monge-Amp\`ere equations and an explicit formula for the Legendre transform of the functional $\mu\rightarrow W^2(\mu,dx)$, where $W^2(\cdot,\cdot)$ is the Wasserstein metric, a distance function on $\M_1(X)$ defined in terms of optimal transport and which turn up in the rate function describing the behaviour of the point process $\Gamma^{(N)}$ as $N\rightarrow \infty$.

\subsection{Kantorovich' Problem of Optimal Transport} 
We will use Kantorovich' formulation (as opposed to Monge's formulation) of the optimal transport problem. The given data is a smooth manifold $Y$, a \emph{cost function} $c:Y\times Y\rightarrow [0,\infty)$, a \emph{source measure}, $\mu\in \M_1(Y)$ and a \emph{target measure}, $\nu\in \M_1(Y)$. Kantorovich problem of optimal transport is the problem of minimizing the functional
$$ C(\gamma) = \int_{Y\times Y} c(x,y) d\gamma(x,y) $$
over the set of \emph{transport plans}, $\Pi(\mu,dx)$, consisting of measures $\gamma\in \M_1(Y\times Y)$ such that the first and second marginals of $\gamma$ equal $\mu$ and $\nu$ respectively. The optimal transport distance between $\mu$ and $\nu$ is the quantity
\begin{equation} \inf_{\gamma\in\Pi(\mu,dx)} C(\gamma). \label{EqTotalCost} \end{equation}
In our case $Y=X$, $\nu=dx$ and $c = d(\cdot,\cdot)^2/2$ where $d$ is the distance function on $X$ induced from $\R^n$. In other words, if $x,y\in \R^n$ and $\pi:\R^n\rightarrow X$ is the quotient map, then
$$ c(\pi x,\pi y) = \frac{d(\pi x,\pi y)^2}{2} = \frac{\inf_{m\in Z^n} |x-y-m|^2}{2}.$$
With this choice of cost function, \eqref{EqTotalCost} is often referred to as the (squared) Wasserstein distance, $W^2(\mu,dx)$, between $\mu$ and $dx$. 

%%%%%%%%%%%%%%%%%%%%%%%%%%%%%%%%%%%%%%%%%%%%%%%%%%%%%%%%%5

\subsection{The $c$-Transform and $c$-Convex Functions}\label{SectPX}
A cost function in optimal transport defines a $c$-transform, closely related to Legendre transform on $\R^n$. Let $C(X)$ be the space of continuous functions on $X$. For $\phi\in C(X)$ the $c$-transform of $\phi$ is
\begin{eqnarray} \phi^c(y)=\sup_{x\in X} -c(x,y)-\phi(x) = \sup_{x\in X} -\frac{d(x,y)^2}{2}-\phi(x) \label{CTransf} \end{eqnarray}
Note that if $\phi$ is a smooth function on $X$ such that $(\phi_{ij}+\delta_{ij})$ is positive definite, then there is a natural way of associating to $\phi$ a convex function on $\R^n$, namely 
\begin{equation} \Phi(x) = \phi(\pi x)+\frac{x^2}{2}. \label{EqPhi} \end{equation}
Let $C(\R^n)$ be the space of continuous functions on $\R^n$ and if $\Phi\in C(X)$ let $\Phi^*$ denote the Legendre transform of $\Phi$. The map from $C(X)$ to $C(\R^n)$ given by $\phi\mapsto \Phi$, relates $c$-transform on $X$ to Legendre transform on $\R^n$ in the sense that
\begin{lemma}\label{LemmaLegCom}
Let $\phi\in C(X)$ and 
$$ \Phi(x) = \phi(\pi x) + \frac{x^2}{2}. $$ 
Then 
$$ \Phi^*(y) = \phi^c(\pi y) + \frac{y^2}{2}. $$
\end{lemma}
\begin{proof}
Note that
\begin{eqnarray} \sup_{x\in \R^n} -\frac{|x-y|^2}{2} - \phi(\pi x) & = & \sup_{x\in [0,1]^n, m\in \Z^n} -\frac{|x-y-m|^2}{2} - \phi(\pi x) \nonumber \\
& = & \sup_{x\in [0,1]^n} -\inf_{m\in \Z^n} \frac{|x-y-m|^2}{2} -\phi(\pi x) \nonumber \\
& = & \sup_{x\in X} -\frac{d(x,\pi y)^2}{2} - \phi(x) \nonumber \\
& = & \phi^c(\pi y). \nonumber 
\end{eqnarray}
This means
\begin{eqnarray} \Phi^*(y) & = & \sup_{x\in \R^n} \bracket{x}{y} - \Phi(x) \nonumber \\
& = & \sup_{x\in \R^n} -\frac{|x-y|^2}{2} - \phi(\pi x) +\frac{y^2}{2} \nonumber \\
& = & \phi^c(\pi y) + \frac{y^2}{2}. \nonumber
\end{eqnarray}
which proves the lemma. 
\end{proof}
It follows that $\phi\in C(X)$ satisfies $(\phi^c)^c=\phi$ if and only if $\Phi$ is convex. The property $(\phi^c)^c=\phi$ is often referred to as $c$-convexity and we will denote the set of functions in $C(X)$ that satisfy this $P(X)$. Since $\Phi^*$ is convex for any $\Phi\in C(\R^n)$ we get that $\phi^c\in P(X)$, for any $\phi\in C(X)$. Moreover, also from the theory of convex functions on $\R^n$, we get that the projection $\phi\mapsto (\phi^c)^c$ of $C(X)$ onto $P(X)$ is monotone in the sense that $(\phi^c)^c(x) \leq \phi(x)$ for all $x\in X$. 

Let $P(\R^n)$ be the set of convex functions on $\R^n$. It is easy to verify that the image of $P(X)$ in $P(\R^n)$ under the map $\phi\mapsto \Phi$ (where $\Phi$ is given by \eqref{EqPhi}) is given by the set
\begin{eqnarray} P_{\Z^n}(\R^n) & = & \{ \Phi\in P(\R^n): \Phi(x+m)-\frac{|x+m|^2}{2} = \Phi(x)-\frac{x^2}{2} \forall m\in \Z^n \} \nonumber \\
& = &  \{ \Phi\in P(\R^n): \Phi(x+m) = \Phi(x)+\bracket{x}{m}+\frac{m^2}{2} \forall m\in \Z^n \} \nonumber \\
& & \label{DefCvx} 
\end{eqnarray}

Now, let $\phi\in P(X)$ and $\Phi$ be the image of $\phi$ in $P_{\Z^n}(\R^n)$. Then $\Phi$ is differentiable at a point $x\in \R^n$ if and only if $\phi$ is differentiable at $\pi x$. Since a convex function on $\R^n$ is differentiable almost everywhere we get that any $\phi\in P(X)$ is differentiable almost everywhere (with respect to $dx$). Further, it follows from \eqref{DefCvx} that $\Phi$ is differentiable at $x$ and $\nabla\Phi(x)=y$ if and only if $\Phi$ is differentiable at $x+m$ and $\nabla\Phi(x+m) = y+m$. This means the map $\nabla \Phi:\R^n\rightarrow \R^n$, where it is defined, factors through to a map $X\rightarrow X$. This map is the so called $c$-gradient map in optimal transport, denoted $\nabla^c \phi$. It satisfies the formula
$$ \nabla^c\phi(\pi x) = \pi \nabla \Phi(x). $$ Further, $\Phi$ is differentiable at $x$ and $\Phi(x)=y$ if and only if $y$ is the unique point in $\R^n$ such that
\begin{equation} \Phi(x) + \Phi^*(y) = \bracket{x}{y}. \label{EqGradDef1} \end{equation}
This holds if and only if
\begin{equation} \phi(\pi x) + \phi^c(\pi y) = -\frac{d(\pi x,\pi y)^2}{2}. \label{EqGradDef2} \end{equation}
We conclude that $\phi$ is differentiable and $\nabla^c \phi(\pi x)=\pi y$ if and only if $\pi y$ is the unique point in $X$ such that \eqref{EqGradDef2} holds. In fact, this is the usual definition of the $c$-gradient and one of its strengths is that it becomes immediately apparent that if $\phi$ is differentiable at $x$ and $\phi^c$ is differentiable at $y=\nabla^c\phi(x)$, then $\nabla^c\phi^c(y)=x$. 

The definition of the Monge-Amp\`ere operator in \eqref{MAOp} makes sense for twice differentiable functions. We will now provide an extension of this operator to $P(X)$. 
\begin{definition}\label{DefWeakOp}
Let $\phi\in P(X)$. We define the Monge-Amp\`ere measure, $\MA(\phi)$, of $\phi$ as
$$ \MA(\phi)= (\nabla^c\phi^c)_* dx. $$
Consequently, we refer to functions in $P(X)$ satisfying 
$$(\nabla^c\phi^c)_* dx = \mu$$
as weak solutions to 
\begin{equation}
\MA(\phi) = \mu \label{MAEqInHom}.
\end{equation}
\end{definition}
Now, the following lemma will serve as a direct justification of Definition~\ref{DefWeakOp} and we will see in Theorem \ref{ThmKantDual} that it fits nicely into the theory of optimal transport. Moreover, weak solutions to \eqref{MAEqGen} in terms of Definition~\ref{DefWeakOp} is the natural analog of so called Alexandrov solutions to Monge-Amp\`ere equations on $\R^n$ (see Section~\ref{SectRegularity}). In fact, we will see in Lemma \ref{LemmaEqLift} that the map $\phi\mapsto \Phi$ where $\Phi$ is given by \eqref{EqPhi} gives a direct link between these two types of solutions.
\begin{lemma}\label{LemmaMASmooth}
Assume $\phi$ is smooth and $(\phi_{ij}+\delta_{ij})$ is strictly positive definite. Then 
$$  \det( \phi_{ij} +\delta_{ij}) dx = (\nabla^c\phi^c)_* dx. $$
\end{lemma}
\begin{proof}
First of all, we claim that $\nabla^c\phi^c:X \rightarrow X$ is one-to-one. To see this, assume that $\nabla^c\phi^c(x_1)=\nabla^c\phi^c(x_2)$ for $x_1,x_2\in X$. Let $\tilde x_1, \tilde x_2\in\R^n$ be lifts of $x_1$ and $x_2$ respectively and $\Phi^*$ be the image of $\phi^*$ in $P_{\Z^n}(\R^n)$. We get
$$ \nabla\Phi^*(\tilde x_1) = \nabla\Phi^*(\tilde x_2) + m. $$
By \eqref{DefCvx} we get $\nabla\Phi^*(\tilde x_1)=\nabla\Phi^*(\tilde x_2+m)$. But since $\phi$, and hence $\Phi$, is smooth $\Phi^*$ must be strictly convex. This means $\tilde x_1=\tilde x_2 + m$ and $x_1=x_2$, proving the claim. 

The previous claim implies, since $\pi\circ \nabla\Phi^*=\nabla^c\phi^c\circ \pi$, that $\pi$ maps $\nabla\Phi^*([0,1)^n)$ diffeomorphically to $X$. Further,
\begin{equation} \det( \phi_{ij} +\delta_{ij})\circ \pi = \det(\Phi_{ij}) = \frac{1}{\det(\Phi^*_{ij})} \label{EqJacDet} \end{equation}
and the numerator of the right hand side of \eqref{EqJacDet} is the Jacobian determinant of the map $\nabla\Phi^*:\R^n\rightarrow \R^n$. Let $h\in C(X)$. Then 
\begin{eqnarray} 
\int_X  h \det(\phi_{ij}+\delta_{ij}) dx &  = &  \int_{\nabla\Phi^*([0,1)^n)} \frac{h \circ \pi}{\det(\Phi^*_{ij})} dx = \int_{[0,1)^n} h\circ \pi \circ \nabla\Phi^*  dx \nonumber \\
& = & \int_{[0,1)^n} h\circ \nabla^c\phi^c \circ \pi dx =  \int_X h\circ \nabla^c\phi^c dx. \label{EqChVar}
\end{eqnarray}
which proves the lemma. 
\end{proof}

%%%%%%%%%%%%%%%%%%%%%%%%%

\subsection{Kantorovich Duality}
We now return to the problem of optimal transport. Although it has very satisfactory solutions providing existence and characterization of minimizers under great generality, we will only give part of that picture here. For us, the important feature of the problem of optimal transport is its dual formulation. Introducing the functional $\xi$ on $C(X)$ defined by
$$ \xi(\phi) = \int_X \phi^c dx $$
we get a functional $J$ on $C(X)$ 
$$ J(\phi)= -\int_X \phi d\mu - \xi(\phi). $$
This functional describes the dual formulation of the problem of optimal transport in the sense that $W^2(\mu,dx)$ can be recovered as the supremum of $J$ over $C(X)$. Moreover, the maximizers of $J$ are weak solutions to a certain Monge-Amp\`ere equation. This is recorded in the following theorem. 
\begin{theorem}[\cite{Kantorovich},\cite{KnottSmith},\cite{Brenier}]\label{ThmKantDual}
Let $\mu\in\M_1(X)$ be absolutely continuous with respect to $dx$. Let $c=d^2/2$ where $d$ is the distance function on $X$ induced from $\R^n$. Then 
\begin{equation} W^2(\mu,dx) = \inf_{\gamma\in \Pi(\mu,dx)} I(\gamma) = \sup_{\phi\in C(X)} J(\phi). \label{KantDualEq} \end{equation}
and there is $\phi_\mu\in P(X)$ such that 
\begin{equation} \sup_{\phi\in C(X)} J(\phi) = J(\phi_\mu). \label{DualOpt} \end{equation}
Moreover,
\begin{equation} \MA(\phi_\mu) =\mu. \label{PushForward} \end{equation} 
\end{theorem}
\begin{remark}
Equation \ref{KantDualEq} is called Kantorovich' duality \cite{Kantorovich} and property \eqref{PushForward} is the Knott-Smith criterion which, in the context of Monge's problem of optimal transport, was discovered independently by Knott and Smith in 1984 \cite{KnottSmith} and by Brenier in 1987 \cite{Brenier}. 
\end{remark}
\begin{proof}[Proof of Theorem \ref{ThmKantDual}]
The theorem is essentially given by Theorem 5.10 in \cite{VillaniOldNew}. As $X$ is a smooth manifold that can be endowed with a complete metric, $X$ is indeed a Polish space. Further, $d$ is continuous and bounded on $X$. Putting $\gamma'=\mu\times dx$ gives 
$$\inf_{\gamma\in \P(\mu,dx)} I(\gamma) \leq I(\gamma')< \infty$$
hence the assumptions in 5.10.i, 5.10.ii and 5.10.iii in \cite{VillaniOldNew} holds. In particular we get that \eqref{KantDualEq} holds and that there is an optimal transport plan $\gamma\in\Pi(\mu,dx)$ and $\phi_\gamma\in P(X)$ such that $\gamma$ is concentrated on the set
\begin{equation} \{(x,y)\in X\times X: \phi_\gamma(x)+\phi_\gamma^c(y) = -c(x,y) \}. \label{DomGamma} \end{equation}
Let $\phi_\mu=\phi_\gamma$. To see that $\eqref{DualOpt}$ holds, note that, since the first and second marginals of $\gamma$ are $\mu$ and $\nu$ respectively, 
\begin{eqnarray} W^2(\mu,dx) & = & \int_{X\times X} c \gamma = -\int_{X\times X} \left(\phi_\mu(x) + \phi_\mu^c(y)\right) \gamma \nonumber \\
& = & -\int_X \phi_\mu(x)d\mu  -\int_X \phi^c_\mu(y) dx. \nonumber \end{eqnarray}
To see that \eqref{PushForward} holds note that $\phi_\mu^c\in P(X)$ is differentiable almost everywhere with respect to $dx$. Let $A\subset X$ be a measurable set and $\dom\nabla^c\phi_\mu^c\subset X$ be the set where $\phi_\mu^c$ is differentiable. We have 
$$\gamma(X\times \dom \nabla^c\phi_\mu^c)=dx(\dom \nabla^c\phi_\mu^c) = 1.$$ 
As $\gamma$ is concentrated on \eqref{DomGamma} we get that $\gamma$ is concentrated on the set 
$$ \{(x,y): y\in \dom \nabla^c\phi_\mu^c, x=\nabla^c\phi_\mu^c(y) \}. $$
This means
\begin{eqnarray} 
\int_{(\nabla^c\phi_\mu^c)^{-1}(A)} dx & = & \int_{X\times (\nabla^c\phi_\mu^c)^{-1}(A)} d\gamma = \int_{A\times (\nabla^c\phi_\mu^c)^{-1}(A)} d\gamma \nonumber \\
 & = & \int_{A\times X} d\gamma = \int_A d\mu, \nonumber
 \end{eqnarray}
in other words $(\nabla^c\phi_\mu^c)_*dx=\mu$, which proves \eqref{PushForward}.
\end{proof}

\subsection{The Variational Approach to Real Monge-Amp\`ere Equations}
We will now reformulate the statement of Theorem \ref{ThmKantDual} in terms of the Legendre transform and Gateaux differentiability of the functional $\xi$. Recall that if $A$ is a functional on $C(X)$, then the \emph{Legendre} transform of $A$ is a functional on the dual vector space of $C(X)$, the space of finite signed measures on $X$, $\M(X)$. This functional is given by
$$ B(\mu)=\sup_{\phi\in C(X)} \int_Y \phi d\mu - A(\phi). $$
Recall also that if $A$ is convex, then $A$ is Gateaux differentiable at a point $\phi$ and has Gateaux differential $\mu$ if $\mu$ is the unique point in $\M(X)$ such that 
$$ B(\mu) = \int_Y \phi d\mu - A(\phi). $$
A priory $W^2(\cdot,dx)$ is defined on $\M_1(X)$. However, we may extend the definition to all of $\M(X)$ by putting $W(\mu,dx)=+\infty$ for any $\mu\notin \M_1(X)$.
We begin with the following lemma
\begin{lemma}\label{XiConv}
The functional $\xi$ is convex on $C(X)$. Moreover, let
$\phi_0,\phi_1\in C(X)$ and
$$ \phi_t=t\phi_1+(1-t)\phi_0. $$
Then, if $\xi(\phi_t)$ is affine in $t$,
$$ \nabla^c\phi_0^c = \nabla^c\phi_1^c $$
almost everywhere with respect to $dx$. 
\end{lemma}
\begin{proof}
First of all, for any $y\in X$, the quantity 
\begin{equation} \phi_t^c(y) = \sup_{x\in X} -c(x,y) - \phi_t(x) \label{PhiC} \end{equation}
is a supremum of functions that are affine in $t$, hence it is convex in $t$. This implies $\xi(\phi_t)$ is convex in $t$. 
Now, assume $\xi(\phi_t)$ is affine in $t$. This implies \eqref{PhiC} is affine in $t$ for almost all $y$. Assume $y$ is a point such that $\nabla^c\phi^c_0(y)$,  $\nabla^c\phi_{1/2}^c(y)$ and $\nabla^c\phi^c_1(y)$ are defined and \eqref{PhiC} is affine. Let $x_{1/2}=\nabla^c\phi_{1/2}^c(y)$. This means 
$$ \phi^c_{1/2}(y) = -c(x_{1/2},y) - \phi_{1/2}(x_{1/2}). $$
By construction 
$$ \phi^c_{t}(y) \geq -c(x_{1/2},y) - \phi_{t}(x_{1/2}) $$
for any $t\in [0,1]$. As $\phi^c_{t}$ and $-c(x_{1/2},y) - \phi_{t}(x_{1/2})$ are affine functions (in $t$) that coincide in one point in the interior of their domains, this inequality implies that they coincide. This means $\nabla^c\phi_0^c(y) = \nabla^c\phi_{1/2}^c(y) = \nabla^c\phi_1^c(y)$. 
As $\nabla^c\phi^c_0$, $\nabla^c\phi_{1/2}^c$ and $\nabla^c\phi_1^c$ are defined almost everywhere, this proves the lemma. 
\end{proof}
This allow us to draw the following conclusions from Theorem \ref{ThmKantDual}
\begin{corollary}\label{CorrDXi}
The functional on $\M(X)$ defined by $\mu\mapsto W^2(-\mu,dx)$ is the Legendre transform of $\xi$. Moreover, for any $\mu\in \M_1(X)$ there is $\phi$ such that
$$ W^2(\mu,dx) + \xi(\phi) = -\int_X\phi d\mu. $$
Finally, $\xi$ is Gateaux differentiable on $C(X)$ and  
\begin{equation} d\xi|_\phi = -\MA(\phi). \label{DXi} \end{equation}
\end{corollary}
\begin{proof}
The first statement is, as long as $\mu\in \M_1(X)$, a direct consequence of \eqref{KantDualEq}. If $\mu\notin\M_1(X)$ then putting $\phi_C=\phi+C$ for some $\phi\in C(X)$ and $C\in \R$ gives $(\phi_C)^c=\phi^c-C$ and
$$ - \int_X \phi_C d\mu - \xi(\phi_C) = -\int_X \phi d\mu - \xi(\phi) + C(1-\mu(X)). $$
Letting $C\rightarrow \infty$ if $\mu(X)<1$ and $C\rightarrow -\infty$ if $\mu(X)>1$ gives 
$$ \sup_{\phi\in C(X)} \phi d\mu - \xi(\phi) = +\infty, $$
proving the first statement. The second statement is also a direct consequence of Theorem \ref{ThmKantDual}. We will now prove that $\xi$ is Gateaux differentiable and that \eqref{DXi} holds. Let $\phi\in C(X)$. We claim that there is $\mu\in \M(X)$ such that 
\begin{equation} \xi(\phi) + W^2(\mu,dx) = -\int \phi d\mu, \label{DXiEq1}\end{equation}
in other words $\mu$ is a supporting hyperplane of $\xi$ at $\phi$. To see this, note that since $W^2(-\cdot,dx)$ is the Legendre transform of $\xi$ we get that $W^2(\cdot,dx)$ is lower semi-continuous and
\begin{equation} \xi(\phi) + W^2(\mu,dx) \geq -\int \phi d\mu \label{XIWInEq}\end{equation}
for all $\mu\in \M(X)$. By lemma~\ref{XiConv}, $\xi$ is convex on $C(X)$. By the involutive property of Legendre transform
$$ \xi(\phi) = \sup_{\mu\in \M(X)} -\int_X \phi d\mu - W^2(\mu,dx). $$
Let $\{\mu_i\} \subset \M(X)$ be a sequence such that 
$$ -\int_X \phi d\mu_i - W^2(\mu_i,dx) \rightarrow \xi(\phi). $$
We may assume, since $W^2(\mu_i,dx)= \infty$ if $\mu_i\notin \M_1(X)$, that $\mu_i\in \M_1(X)$ for all $i$. Since $\M_1(X)$ is compact we may take a subsequence $\{\mu_{i_k}\}$ converging to some $\mu\in \M_1(X)$. By the lower semi-continuity of $W^2(\cdot,dx)$ we get 
$$  - \int_X\phi\mu -W^2(\mu,dx) \geq \liminf_{k\rightarrow \infty} -\int_X\phi_0 \mu_{i_k} -W^2(\mu_{i_k},dx)  = \xi(\phi_0). $$
which, together with \eqref{XIWInEq}, proves the claim. We will now prove that this implies 
\begin{equation} (\nabla^c\phi^c)_* dx = \mu. \label{DXiEq2} \end{equation} 
As this relation determines $\mu$ we get that $\mu$ must be the unique supporting hyperplane at $\phi$. This implies $\xi$ is Gateaux differentiable at $\phi$ and $d\xi_\phi = \mu$, proving the second statement in the corollary. 

Now, to see that \eqref{DXiEq2} holds, note that \eqref{DXiEq1} implies $W^2(\mu,dx)<\infty$ and hence $\mu\in M_1(X)$. By Theorem \ref{ThmKantDual} there is a function $\phi_\mu\in P(X)$ such that $\MA(\phi_\mu) = \mu$ and 
$$ W^2(\mu,dx) + \xi(\phi_\mu) = -\int \phi_\mu d\mu. $$
This means $\mu$ is a supporting hyperplane of $\xi$ both at $\phi$ and at $\phi_\mu$. This implies $\xi(t\phi + (1-t)\phi_\mu)$ is affine. By Lemma~\ref{XiConv}, $\nabla^c\phi^c$ and $\nabla^c\phi_\mu^c$ coincide almost everywhere with respect to $dx$ and hence \eqref{DXiEq2} holds. 
\end{proof}

%%%%%%%%%%%%%%%%%%%%%%%%%%%%%%%%%%%%%%%%%%%%%%%%%%%%%%

%%%%%%%%%%%%%%%%%%%%%%%%%%%%%%%%%%%%%%%%%%%%%%%%%%%%%%

%%%%%%%%%%%%%%%%%%%%%%%%%%%%%%%%%%%%%%%%%%%%%%%%%%%%%%

\section{A Large Deviation Principle}\label{SectionLDP}
This section is devoted to Theorem~\ref{MainThmLDP} which will be the key part in the proof of Theorem \ref{MainThmGen}. Before we state Theorem~\ref{MainThmLDP} we will recall the definition of the relative entropy function. 
\begin{definition}
Assume $\mu,\mu_0\in \M(X)$ and, if $\mu$ is absolutely continuous with respect to $\mu$, let $\mu/\mu_0$ denote the density of $\mu$ with respect to $\mu_0$. The \emph{relative entropy} of $\mu$ with respect to $\mu_0$ is  
$$ Ent_{\mu_0}(\mu) = \begin{cases} \int_X \mu \log \frac{\mu}{\mu_0} & \textnormal{if $\mu$ is a probability measure and absolutely} \\
& \textnormal{continuous with respect to $\mu_0$} \\
+\infty & \textnormal{otherwise,} \end{cases} $$
\end{definition}
We recall the basic property that $Ent_{\mu_0}(\mu)\geq 0$ with equality if and only if $\mu=\mu_0$.
\begin{theorem}\label{MainThmLDP}
Let $\mu_0\in \mathcal{M}_1(X)$ be absolutely continuous and have positive density with respect to $dx$. Let $\beta \in \R $. Assume $\Gamma_\beta^{(N)}$ is defined as in section~\ref{SectPointProc}. Then
$$ \left\{\Gamma^{(N)}_\beta\right\} $$
satisfy a \emph{Large Deviation Principle} with rate $r_N=N$ and rate function 
$$ G(\mu) = \beta W^2(\mu,dx) + Ent_{\mu_0}(\mu) + C_{\mu_0, \beta} $$
where $W^2(\mu,dx)$ is the squared Wasserstein 2-distance between $dx$ and $\mu_0$ (defined in the previous section) and $C_{\mu_0,\beta}$ is a constant ensuring $\inf_{\M_1(X)} G = 0$.
\end{theorem}
%\begin{remark}
%Motivated by a theormodynamic interpretation (see Section \ref{SectGenCase}) we will refer to the rate function in the theorem above as the \emph{Gibbs Free Energy} of the system defined by the data $(\beta,\mu_0)$. 
%\end{remark}
Before we move on we will recall the definition of a Large Deviation Principle. 
\begin{definition}\label{DefLDP}
Let $\chi$ be a topological space, $\{\Gamma_N\}$ a sequence of probability measures on $\chi$, $G$ a lower semi continuous function on $\chi$ and $r_N$ a sequence of numbers such that $r_N \rightarrow \infty$. Then $\{\Gamma_N\}$ satisfies a \emph{large deviation principle} with rate function $G$ and rate $r_N$ if, for all measurable $E \subset \chi$,
$$ -\inf_{E^\circ} G \leq \liminf_{N\rightarrow \infty} \frac{1}{r_N}\log \Gamma_N(E) \leq \limsup_{N\rightarrow \infty} \frac{1}{r_N} \log \Gamma_N(E)\leq -\inf_{\bar E} G $$
where $E^\circ$ and $\bar E$ are the interior and the closure of $E$. 
\end{definition}
In our case $\chi=\M_1(X)$. As we may endow $\M_1(X)$ with the Wasserstein 1-metric, metricizing the topology of weak* convergence on $\chi$, we may think of $\M_1(X)$ as a metric space. Further, by Prohorov's Theorem, $\M_1(X)$ is compact. In this setting there is an alternative, and well known, criteria for when a large deviation principle exist.
\begin{lemma}\label{LemmaLDP}
Let $\chi$ be a compact metric space, $\{\Gamma_N\}$ a sequence of probability measures on $\chi$, $G$ a function on $\chi$ and $r_N$ a sequence of numbers such that $r_N \rightarrow \infty$. Let $B_d(\mu)$ denote the open ball in $\chi$ with center $\mu$ and radius $d$. Then $\{\Gamma_N\}$ satisfies a \emph{large deviation principle} with rate function $G$ and rate $r_N$ if and only if, for all $\mu\in \chi$
\begin{eqnarray}
G(\mu) & = & \lim_{\delta\rightarrow 0} \limsup_{N\rightarrow \infty} -\frac{1}{r_N} \log \Gamma_N(B_\delta(\mu)) \nonumber \\
& = & \lim_{\delta\rightarrow 0} \liminf_{N\rightarrow \infty} -\frac{1}{r_N} \log \Gamma_N(B_\delta(\mu)) \nonumber 
\end{eqnarray}
\end{lemma}
\begin{proof}
Let $\mathcal{B}$ be the basis of the topology on $\chi$ given by
$$ \mathcal{B} = \{B_d(\mu): d>0, \mu\in \chi\}. $$
By Theorem 4.1.11, Theorem 4.1.18 and Lemma 1.2.18 (recall that $\chi$ is compact by assumption) in \cite{DemboZeitouni}, $\{\Gamma_N\}$ satisfies a large deviation principle with rate function $G$ and rate $r_N$ if and only if
\begin{eqnarray}
G(\mu) & = & \sup_{B\in \mathcal{B} : \mu\in B} \limsup_{N\rightarrow \infty} -\frac{1}{r_N} \log \Gamma_N(B_\delta(\mu)) \nonumber \\
& = & \sup_{B\in \mathcal{B}: \mu\in B} \liminf_{N\rightarrow \infty} -\frac{1}{r_N} \log \Gamma_N(B_\delta(\mu)). \nonumber 
\end{eqnarray}
Now, if $\mu\in B\in \mathcal{B}$ then $B_d(\mu)\subset B$ for $d$ small enough. This means, since
\begin{equation} \lim_{d\rightarrow 0} \limsup_{N\rightarrow \infty} -\frac{1}{r_N} \log \Gamma_N(B_\delta(\mu)) \label{LDPEquiv1} \end{equation}
is increasing as $d\rightarrow 0$, that
\begin{equation}  \eqref{LDPEquiv1} \geq \sup_{B\in \mathcal{B}: \mu\in B} \limsup_{N\rightarrow \infty} -\frac{1}{r_N} \log \Gamma_N(B_\delta(\mu)). \label{LDPEquiv2} \end{equation}
Since, for any $d>0$, $B_d(\mu)$ is a candidate for the supremum in the right hand side of \eqref{LDPEquiv2} we get that equality must hold in \eqref{LDPEquiv2}. The same argument goes through with $\limsup$ replaced by $\liminf$. This proves the lemma.
\end{proof}

Finally we recall the well known
\begin{theorem}[Sanov's theorem, see for example 6.2.10 in \cite{DemboZeitouni}]
Let $\mu_0\in \M_1(X)$. Then the family
$$ \left\{\left(\delta^{(N)}\right)_* \mu_0^{\otimes N}\right\} $$
satisfies a large deviation principle with rate $r_N=N$ and rate function $Ent_{\mu_0}$. 
\end{theorem}

%%%%%%%%%%%%%%%%%%%%%%%%%%%%%%%%%%%%%%%%%%%%%%%%%%%%%%

\subsection{The Zero Temperature Case and the G\"artner-Theorem}\label{SectBetaInf}
Recall that $N=k^n$. For each $\beta\in \R$ we get a family of probability measures $\{\Gamma^{(N)}_\beta\}_{k\in \N}$. Theorem~\ref{MainThmGen} and Theorem~\ref{MainThmLDP} are both concerned with the behavior of these families. In this section we will consider the family $\{\Gamma^{(N)}_k\}_{k\in \N}$. We will prove a large deviation principle for this family (see Theorem~\ref{ThmBetaInf}) which, in Section~\ref{SectGenCase}, will be used to prove Theorem \ref{MainThmLDP}. 
\begin{theorem}\label{ThmBetaInf}
Let $\mu_0\in \mathcal{M}_1(X)$ be absolutely continuous and have positive density with respect to $dx$. Assume $\Gamma_\beta^{(N)}$ is defined as in section~\ref{SectPointProc}. Then
$$ \{\Gamma^{(N)}_k\} $$
satisfies a large deviation principle with rate $r_N=kN$ and rate function 
$$ G(\mu) = W^2(\cdot,dx). $$
\end{theorem} 
Recall that if $\Gamma$ is a probability measure on a topological vector space $\chi$, then the moment generating function of $\Gamma$ is the functional on the dual vector space $\chi^*$ given by
$$ Z_\Gamma(\phi) = \int_\chi e^{-\bracket{\phi}{\mu}}d\Gamma(\mu) $$  
where $\bracket{\cdot}{\cdot}$ is the pairing of $\chi$ and $\chi^*$.
The significance of this for our purposes lies in the G\"artner-Ellis theorem. Before we state this theorem, recall that a sequence of (Borel) probability measures $\{\Gamma_N\}$ on a space $\chi$ is \emph{exponentially tight} if for each $\epsilon\in\R$ there is a compact $K_\epsilon\subseteq \chi$ such that for all $N$
\begin{equation} \limsup_{\N\rightarrow \infty} \frac{1}{N}\log \Gamma_N(\chi\setminus K_\epsilon) \leq \epsilon. \label{ExpTight} \end{equation}
In our case, when $\chi$ is compact, this is automatically satisfied since choosing $K_\epsilon=\chi$ for any $\epsilon$ gives that the left hand side of \eqref{ExpTight} is $-\infty$ for all $N$. 
\begin{theorem}[The G\"artner-Ellis Theorem. See for example Corollary 4.5.27 in \cite{DemboZeitouni}]
Let $\chi$ be a locally convex topological vector space, $\{\Gamma_N\}$ an exponentially tight sequence of probability measures on $\chi$ and $r_N$ a sequence such that $r_N\rightarrow \infty$. Let $Z_{\Gamma_N}$ be the moment generating function of $\Gamma_N$ and assume
$$ F(\phi) = \lim_{N\rightarrow \infty} \frac{1}{r_N}\log Z_{\Gamma_N}(r_N \phi) $$
exist, is finite valued, lower semi continuous and Gateaux differentiable. Then $\Gamma_N$ satisfies a large deviation principle with rate $r_N$ and rate function given by the Legendre transform of $F$. 
\end{theorem}
Theorem \ref{ThmBetaInf} will follow from the G\"artner-Ellis theorem and the crucial point will be the following lemma.
\begin{lemma}\label{LemmaHNXi}
Let $\mu_0\in \mathcal{M}_1(X)$ be absolutely continuous and have positive density with respect to $dx$. Assume $\Gamma_\beta^{(N)}$ is defined as in section~\ref{SectPointProc}. Then
$$ \lim_{N\rightarrow \infty}\frac{1}{kN}\log Z_{\Gamma^{(N)}}(kN\phi) = \xi(-\phi). $$
\end{lemma}
\begin{proof}
Note that if $\mu_N$ is a measure on $X^N$ and $F$ is a function on $\M_1(X)$, then, since $\Gamma^{(N)}=(\delta^{(N)})_*\mu_{\beta_N}^{(N)}$, 
$$\int_{\M_1(X)} F(\mu) \Gamma^{(N)} = \int_{X^N} F\left(\delta^{(N)}(x)\right) d\mu^{(N)}_{\beta_N}.$$
Moreover,  
$$ \bracket{kN \phi}{\delta^{(N)}(x)} = kN\int_X \phi \frac{1}{N} \sum \delta_{x_i} = k\sum \phi(x_i). $$
This means 
$$ Z_{\Gamma^{(N)}}(kN\phi) = \int_{\M_1(X)} e^{\bracket{r_N\phi}{\mu}} \Gamma^{(N)} = \int_{X^N} e^{k\sum \phi(x_i) } d\mu_{\beta_N}^{(N)}.  $$
Using the symmetries in the explicit form of $\mu_{\beta_N}^{(N)}$ we get
\begin{eqnarray}
Z_{\Gamma^{(N)}}(kN\phi) 
& = & \int_{X^N} \sum_{\sigma}\prod_i \Psi^{(N)}_{p_i}(x_{\sigma(i)}) e^{k\phi(x_{\sigma(i)})} d\mu_0^{\otimes N} \nonumber \\
& = & \sum_{\sigma}\int_{\sigma^{-1}(X^N)} \prod_i \Psi^{(N)}_{p_i}(x_i) e^{k\phi(x_i)} d\mu_0^{\otimes N} \nonumber \\
& = & N! \int_{X^N} \prod_i \Psi^{(N)}_{p_i}(x_i) e^{k\phi(x_i)} d\mu_0^{\otimes N} \nonumber \\
& = &  N! \prod_i \int_{X} \Psi^{(N)}_{p_i}(x) e^{k\phi(x)} d\mu_0 \label{EqTheorem0TempLimit}
\end{eqnarray}
Introducing the notation 
$$ c^{(N)}_{p} = -\frac{1}{k}\log \Psi^{(N)}_{p} $$
we get 
\begin{equation} Z_{\Gamma^{(N)}}(kN\phi) = N! \prod_{p\in \frac{1}{k}\Z^n/\Z^n} \int_{X} e^{k(-c^{(N)}_p+\phi)} d\mu_0. \label{EqZN2} \end{equation}

Now, we claim that 
\begin{equation} c^{(N)}_p\rightarrow d(x,p)^2/2 \label{CNConv} \end{equation} 
uniformly in $p$ and $x$.
To see this, note first that 
$$d(x,p)^2 = \inf_{m\in \Z^n+p} |x-m|^2$$ 
and 
\begin{eqnarray} 
c_p^{(N)}(x) & = & -\frac{1}{k} \log \sum_{m\in \Z^n+p} e^{-k|x-m|^2/2} \nonumber \\
& \leq & -\frac{1}{k} \log \sup_{m\in \Z^n+p} e^{-k|x-m|^2/2} = \inf_{m\in \Z^n+p} |x-m|^2/2. \nonumber 
\end{eqnarray}
On the other hand, by the exponential decay of $e^{-|x-m|^2}$ there is a large constant, $C$, (independent of $x$ and $p$) such that
$$ \sum_{m\in \Z^n+p} e^{-k|x-m|^2/2} \leq C\sup e^{-k|x-m|^2/2} $$
and
\begin{eqnarray} 
c_p^{(N)}(x) & = & -\frac{1}{k} \log \sum_{m\in \Z^n+p} e^{-k|x-m|^2/2} \geq -\frac{1}{k} \log \left(C\sup_{m\in \Z^n+p} e^{-k|x-m|^2/2}\right) \nonumber \\
& = & -\frac{\log C}{k} + \inf_{m\in \Z^n+p} |x-m|^2/2. \nonumber 
\end{eqnarray}
This proves the claim. We claim further that
\begin{equation} \frac{1}{k}\log \int_X e^{k(-c_p^{(N)}+\phi)}  d\mu_0 \rightarrow (-\phi)^c(p) \label{PhiCClaim} \end{equation}
uniformly in $p$. 
To see this, note first that \eqref{CNConv} together with the fact that the family $\{d(\cdot,p)^2/2: p\in X\}$ is equi-continuous implies that 
$$\{c_p^{(N)}:k\in \N,p\in X\}$$ 
is equi-continuous. This means for any $\epsilon>0$ there is $d>0$ such that for all $k\in\N$ and $p,x_*\in X$
\begin{equation} |c^{(N)}_p(x) - \phi(x) - (c_p^{(N)}(x_*) - \phi(x_*))|\leq \epsilon \label{ThmBetaInfEq1} \end{equation}
as long as $x\in B_d(x_*)$.
Further, as $\mu_0$ has full support, is absolutely continuous and has smooth density with respect to $dx$ there is a large constant $C$ such that 
\begin{equation} C\mu_0(B_d(x_*)) \geq 1 \label{ThmBetaInfEq2} \end{equation}
for all $x_*\in X$. We get trivially
\begin{eqnarray} \frac{1}{k}\log \int_X e^{k(-c^{(N)}_{p}+\phi)}  d\mu_0  & \leq &  \frac{1}{k}\log \sup_{x\in X} e^{k(-c^{(N)}_{p}+\phi)} \nonumber \\
& = & \sup -c^{(N)}_p(x) +\phi(x) \label{ThmBetaInfEq3} 
\end{eqnarray}
For each $N$, let $x^{(N)}_*$ satisfy 
$$-c^{(N)}_p(x^{(N)}_*) +\phi(x^{(N)}_*) = \sup_{x\in X} -c^{(N)}_p(x) +\phi(x). $$ 
Using \eqref{ThmBetaInfEq1} and \eqref{ThmBetaInfEq2} gives
\begin{eqnarray} 
\frac{1}{k}\log \int_X e^{k(-c^{(N)}_p+\phi)}  d\mu_0 & \geq & \frac{1}{k}\log \int_{B_\delta(x_*^{(N)})} e^{k(\sup_{x\in X} -c^{(N)}_p+\phi-\epsilon)}  d\mu_0 \nonumber \\ 
& = &\frac{1}{k}\log \int_{B_\delta(x_*^{(N)})} d\mu_0 + \sup_{x\in X}-c^{(N)}_p(x) + \phi(x)-\epsilon \nonumber \\
& \geq & \frac{1}{k}\log \frac{1}{C} \int_X d\mu_0 +\sup_{x\in X}-c^{(N)}_p(x) + \phi(x)-\epsilon. \label{ThmBetaInfEq4}
\end{eqnarray}
Finally, letting $k,N\rightarrow \infty$ and $\epsilon\rightarrow 0$ in \eqref{ThmBetaInfEq3} and \eqref{ThmBetaInfEq4} proves \eqref{PhiCClaim}. Recalling equation \eqref{EqZN2}, we have
\begin{eqnarray} \frac{1}{kN}\log Z_{\Gamma^{(N)}}(kN\phi) & = & \frac{1}{kN} \log N! \prod_{p\in \frac{1}{k}\Z^n/\Z^n} \int_{X} e^{k(-c^{(N)}_p+\phi)} d\mu_0 \nonumber \\
& = & \frac{\log N!}{kN}+\frac{1}{N}\sum_{p\in \frac{1}{k}\Z^n/\Z^n} \frac{1}{k}\log \int_{X} e^{k(-c^{(N)}_p +\phi)} d\mu_0 \label{LogZN}
\end{eqnarray} 
By Sterling's formula, $\log N! \leq N\log N+O(\log N)$. This means, since $N=k^n$, that
the first term in \eqref{LogZN} is bounded by $(\log k^n)/k+O(\log k^n)/k^{n+1}$ which vanishes as $k\rightarrow \infty$. Finally, using \eqref{PhiCClaim} we get, since $\frac{1}{N}\sum_{p\in \frac{1}{k}\Z^n/\Z^n} \delta_p\rightarrow dx$ in the weak* topology, that the second term converges to  
$$\int (-\phi)^c(p) dx = \xi(-\phi). $$
This proves the lemma.
\end{proof}
When proving Theorem~\ref{ThmBetaInf} we will also need the following lemma.
\begin{lemma}\label{LemmaXiCont}
The functional $\xi$ is continuous on $C(X)$. 
\end{lemma}
\begin{proof}
We will prove that for any $\phi_0,\phi_1\in C(X)$
\begin{equation} \sup_X |\phi_0^c-\phi_1^c| \leq \sup_X |\phi_1-\phi_0|. \label{PhiPhiC} \end{equation}
Once this is established the lemma follows from the dominated convergence theorem. 
To see that \eqref{PhiPhiC} holds, let $y\in X$. By compactness and continuity there is $x_y\in X$ such that 
$$ \phi_0^c(y) = \sup_{x\in X} -c(x,y) +\phi_0(x) = -c(x_y,y)-\phi_0(x_y). $$
By construction
$$ \phi_1^c(y) = \sup_{x\in X} -c(x,y) +\phi_1(x) \geq -c(x_y,y)-\phi_1(x_y). $$
We get
$$ \phi_0^c(y)-\phi_1^c(y) \leq \phi_1(x_y)-\phi_0(x_y) \leq \sup_X |\phi_1-\phi_0|. $$
By interchanging the roles of $\phi_0$ and $\phi_1$ we get
$$ \phi_1^c(y)-\phi_0^c(y) \leq \sup_X |\phi_1-\phi_0| $$
and hence that \eqref{PhiPhiC} holds.
\end{proof} 
\begin{proof}[Proof of Theorem \ref{ThmBetaInf}]
We want to apply the G\"artner-Ellis theorem. As $\chi = M_1(X)$ is compact, tightness of $\Gamma^{(N)}$ holds automatically. By Lemma \ref{LemmaHNXi} 
$$ \lim_{N\rightarrow \infty} \frac{1}{r_N} \log \Lambda_{\Gamma^{(N)}_k}(r_N\phi)=\xi(-\phi). $$
Further, $\xi$ is finite valued since $\phi^c$ is continuous, and hence bounded, for any $\phi\in C(X)$. By Lemma~\ref{LemmaXiCont}, $\xi$ is continuous. Finally, by Corollary~\ref{CorrDXi}, $\xi$ is Gateaux differentiable. As $W^2(-\cdot,dx)$ is the Legendre transform of $\xi$, and hence $W^2(\cdot,dx)$ is the Legendre transform of $\xi(-\cdot)$, the theorem follows from the G\" artner-Ellis theorem. 
\end{proof}

%%%%%%%%%%%%%%%%%%%%%%%%%%%%%%%%%%%%%%%%%%%%%%%%%%%%%

\subsection{A Thermodynamic Interpretation and Reduction to the Zero Temperature Case}\label{SectGenCase}
The proof of Theorem \ref{MainThmLDP} is based on a result on large deviation principles for Gibbs measures. Because of this we explain in this section how $\{\mu_\beta^{(N)}\}$ can be seen as the Gibbs measures of certain thermodynamic systems. If we introduce the $N$-particle Hamiltonian
$$ H^{(N)}(x_1, \ldots, x_N) = -\frac{1}{k}\log \perm(\Psi_{p_i}(x_j)) $$
we may write $\mu^{(N)}_\beta$ on the form 
$$ \mu_\beta^{(N)} = e^{-\beta H^{(N)}}d\mu_0^{\otimes N}. $$
This means $\mu_\beta^{(N)}$ admits a thermodynamic interpretation as the \emph{Gibbs measure}, or \emph{canonical ensemble}, of the system determined by the Hamiltonian $H^{(N)}$ and the background measure $\mu_0$. In this interpretation $\mu_\beta^{(N)}$ is the equilibrium state of the system when the temperature is assumed fixed at $Temp=1/\beta$ and Theorem \ref{ThmBetaInf} is describing the zero-temperature limit. Theorem~\ref{MainThmLDP} will follow from Theorem~\ref{ThmBetaInf} and a theorem on equi-continuous and uniformly bounded Hamiltonians. To state that theorem we need to define what it means for the family $\{\frac{H^{(N)}}{N}\}$ to be equi-continuous. Let $d(\cdot,\cdot)$ be the distance function induced by the standard Riemannian metric on $X$. This defines distance functions, $d^{(N)}(\cdot,\cdot)$, on $X^N$ given by
\begin{equation} d^{(N)}(x,y) = d^{(N)}(x_1,\ldots, x_N,y_1,\ldots,y_N) = \frac{1}{N}\inf_{\sigma}\sum_i d(x_i,y_{\sigma(i)}) \label{DistDef} \end{equation}
where the infimum is taken over all permutations $\sigma$ of the set $\{1,\ldots, N\}$. We will say that the family of functions $\frac{H^{(N)}}{N}$ on $X^N$ is (uniformly) equi-continuous if for every $\epsilon>0$ there is $d>0$ such that for all $N$
\begin{equation} \left|\frac{1}{N}H^{(N)}(x)-\frac{1}{N}H^{(N)}(y)\right|\leq \epsilon \label{EquiContEq} \end{equation}
whenever $d^{(N)}(x,y)\leq d$. 
Before we move on to state the Theorem~\ref{ThmRedToZeroTemp} we prove the following well known lemma.
\begin{lemma}
Let $x=(x_1,\ldots,x_N)\in X^N$ and $y=(y_1,\ldots, y_N)\in X^N$. Then \eqref{DistDef}
is the optimal transport cost 
with respect to the cost function $d(\cdot,\cdot)$, of transporting the measure $\delta^{(N)}(x) = \frac{1}{N}\sum \delta_{x_i}$ to the measure $\delta^{(N)}(y)=\frac{1}{N}\sum \delta_{y_i}$.
\end{lemma}
\begin{proof}
We need to prove that 
\begin{equation} \eqref{DistDef} = \inf_\gamma \int_{X\times X} d(x,y) \gamma \label{DistCost} \end{equation}
where the infimum is taken over all $\gamma\in \M_1(X\times X)$ with first and second marginal given by $\delta^{(N)}(x)$ and $\delta^{(N)}(y)$ respectively. We will refer to any $\gamma\in \M_1(X\times X)$ satisfying this as a feasible transport plan. The conditions on the marginals imply that any feasible transport plan is supported on the intersection of the sets $\{x_i\}\times X$ and $X\times \{y_i\}$, in other words on the set $\{x_i\}\times\{y_i\}$. We conclude that the set of feasible transport plans is given by
\begin{equation} \left\{\sum_{i,j} a_{ij} \delta_{(x_i,y_j)}: a_{ij}\geq 0, \sum_i a_{ij} = 1/N, \sum_j a_{ij} = 1/N\right\}, \label{Polytope} \end{equation}
in other words a polytope in $\M_1(X\times X)$. It follows that the infimum in \eqref{DistCost} is attained on one or more of the vertices of \eqref{Polytope}. Moreover, any permutation, $\sigma$, of $N$ elements induce a feasible transport plan
$$ \gamma_\sigma = \frac{1}{N}\sum_i \delta_{(x_i,y_{\sigma(i)})} $$
with transport cost 
$$ \int_{X\times X} d(x,y) \gamma_\sigma = \frac{1}{N}\sum_i d(x_i,y_{\sigma(i)}). $$
It is easy to verify that any vertex of \eqref{Polytope} occur as $\gamma_\sigma$ for some permutation $\sigma$. This proves the lemma. 
\end{proof}
Note that this lemma implies that if we equip $\M_1(X)$ with the Wasserstein 1-metric, which metricizes the weak* topology on $\M_1(X)$, then the distance function defined in \eqref{DistDef} makes the embeddings $$\delta^{(N)}:X^N\hookrightarrow \mathcal{M}_1(X)$$
isometric embeddings.

\begin{theorem}[\cite{EllisEtAl}]\label{ThmRedToZeroTemp}
Assume $X$ is a compact manifold, $\mu_0\in \mathcal{M}_1(X)$, $\{\frac{H^{(N)}}{N}\}$ is a uniformly bounded and equi-continuous family of functions on $X^N$ and $\beta_N$ is a sequence of numbers tending to infinity. Assume also that 
$$ \left(\delta^{(N)}\right)_* e^{-\beta_N H^{(N)}}d\mu_0^{\otimes N} $$ 
satisfies a Large Deviation Principle with rate $N\beta_N$ and rate function $E$. Then, for any $\beta\in \R$, 
$$ \left(\delta^{(N)}\right)_* e^{-\beta H^{(N)}}d\mu_0^{\otimes N} $$ 
satisfies a Large Deviation Principle with rate $N$ and rate function $\beta E + Ent_{\mu_0}$. 
\end{theorem}
For completeness, we will include a proof of Theorem~\ref{ThmRedToZeroTemp} here. It will be based on the following
\begin{proposition}[\cite{EllisEtAl}]\label{TheoremThmRedToZeroTempHamConv}
Assume $X$ is a compact manifold, $\mu_0\in \mathcal{M}_1(X)$, $\beta\in \R$, $\{\frac{H^{(N)}}{N}\}$ is a  family of functions on $X^N$. Assume also that there is a functional $E$ on $\mathcal{M}_1(X)$ satisfying
\begin{equation} \sup_{X^N} \left|\frac{H^{(N)}}{N} - E\circ \delta^{(N)}\right| \rightarrow 0 \label{ThmRedToZeroTempHamConvEq}\end{equation}
as $N\rightarrow \infty$. Then 
$$ \left(\delta^{(N)}\right)_* e^{-\beta H^{(N)}}\mu_0^{\otimes N} $$
satisfies a Large Deviation Principle with rate $N$ and rate function $\beta E + Ent_{\mu_0}$.
\end{proposition}
\begin{proof}
Let $\mu\in \M_1(X)$ and $B_d(\mu)$ be the ball of (Wasserstein-1) radius $d$ centred at $\mu$ and 
$$ B^{(N)}_d(\mu) = (\delta^{(N)})^{-1}(B_d(\mu)) \subset X^N. $$  
Using \eqref{ThmRedToZeroTempHamConvEq} we get
\begin{eqnarray}
 & & \lim_{d\rightarrow 0} \liminf_{N\rightarrow \infty} -\frac{1}{N}(\delta^{(N)})_*e^{-\beta H^{(N)}}\mu_0^{\otimes N}(B_d(\mu))  \nonumber \\
 & = & \lim_{d\rightarrow 0} \liminf_{N\rightarrow \infty} -\frac{1}{N}\log\int_{B^{(N)}_d(\mu)} e^{-\beta H^{N}(x)}d\mu_0^{\otimes N}  \nonumber \\
 & = & \lim_{d\rightarrow 0} \liminf_{N\rightarrow \infty} -\frac{1}{N}\log\int_{B^{(N)}_d(\mu)} e^{-\beta N (E\circ \delta^{(N)}(x)+o(1))}d\mu_0^{\otimes N}  \nonumber \\
& = & \beta E(\mu) + \lim_{d\rightarrow 0} \liminf_{N\rightarrow \infty} -\frac{1}{N}\log\int_{B^{(N)}_d(\mu)} d\mu_0^{\otimes N}. \label{EqLDP}
\end{eqnarray}
and similarily with $\liminf$ replaced by $\limsup$ (here $o(1)\rightarrow 0$ uniformly in $x$ as $N\rightarrow \infty$). By Sanov's theorem $(\delta^{(N)})_*\mu_0^{\otimes N}$ satisfies a large deviation principle with rate $N$ and rate function $Ent_{\mu_0}$. Hence, by Lemma \ref{LemmaLDP}, the second term in \eqref{EqLDP} is $Ent_\gamma(\mu)$. Using Lemma~\ref{LemmaLDP} again, this proves the proposition.
\end{proof}
It turns out that in the compact setting, under the assumptions of uniform boundedness and equi-continuity, the assumption of convergence in Proposition \ref{TheoremThmRedToZeroTempHamConv} always holds for some functional $U$ on $\mathcal{M}_1(X)$.
\begin{lemma}\label{LemmaHamiltonianConv}
Assume $X$ is a compact manifold, $\mu_0\in \mathcal{M}_1(X)$ and $\{\frac{H^{(N)}}{N}\}$ is a uniformly bounded and equi-continuous family of functions on $X^N$. Then there is a function $U$ on $\M_1(X)$ such that, after possibly passing to a subsequence,
\begin{equation} \sup_{X^N} |\frac{H^{(N)}(x)}{N}-U\circ \delta^{(N)}(x)| \rightarrow 0 \label{EqLemmaAbstractConvergence} \end{equation}
as $N\rightarrow \infty$.
\end{lemma}
\begin{proof}
Using the embeddings $\delta^{(N)}:X^n\hookrightarrow \M_1(X)$ the functions $H^{(N)}$ define a sequence of functionals, $\mathcal{H}^{(N)}$, defined on the subspaces $\delta^{(N)}(X^N)\subset \M_1(X)$. By a standard procedure (we will explain it below) it is possible to define an equi-continuous family of extensions, $\{U^{(N)}\}$, of $\frac{\mathcal{H}^{(N)}}{N}$ on $\mathcal{M}_1(X)$. By Arzel\`a-Ascoli theorem $U^{(N)}$, after possibly passing to a subsequence, will converge to a functional $U$ satisfying \eqref{EqLemmaAbstractConvergence}.
We may define the extensions $U^{(N)}$ in the following way: Note that by assumption the functions $\frac{H^{(N)}}{N}$ all satisfy the same modulus of continuity, $\omega$. We define $U^{(N)}:\M_1(X)\rightarrow \R$ as
$$ U^{(N)}(\mu)=\inf_{\nu\in \delta^{(N)}(X^N)} \frac{\mathcal{H}^{(N)}(\nu)}{N}+\omega(d(\mu,\nu)) $$ 
where $d(\cdot,\cdot)$ is the Wasserstein 1-distance on $\M_1(X)$. It follows from the definition of moduli of continuity that $U^{(N)}=\frac{\mathcal{H}^{(N)}}{N}$ on $\delta^{(N)}(X^N)$. As $\M_1(X)$ is compact we may take $\omega$ to be sub-additive. It follows that the function $\omega(d(\mu,\cdot))$ satisfies $\omega$ as modulus of continuity. This means $U^{(N)}$, being a supremum of functions satisfying $\omega$, also satisfy $\omega$. In particular the family $\{U^{(N)}\}$ is equi-continuous. 
\end{proof}

We can now prove Theorem \ref{ThmRedToZeroTemp}.
\begin{proof}[Proof of Theorem \ref{ThmRedToZeroTemp}]
As above, let $B^{(N)}_d(\mu) = (\delta^{(N)})^{-1} (B_d(\mu)) \subset X^N$, where $B_d(\mu)$ is the ball in $\M_1(X)$ centered at $\mu$ with radius $d$. By the assumed Large Deviation Principle and Lemma \ref{LemmaLDP}, for any $\mu\in \mathcal{M}_1(X)$,
$$ E(\mu) = \lim_{d\rightarrow 0} \liminf_{N\rightarrow \infty} - \frac{1}{N\beta_N}\log \int_{B^{(N)}_d(\mu)} e^{-\beta_N H^N}\mu_0^{\otimes N}. $$
On the other hand, by Lemma \ref{LemmaHamiltonianConv} there is a function $U$ on $\M_1(X)$ such that, after possibly passing to a subsequence, \eqref{EqLemmaAbstractConvergence} holds. This means
\begin{eqnarray}
E(\mu) & = & \lim_{\delta\rightarrow 0} \liminf_{N\rightarrow \infty} -\frac{1}{N\beta_N}\log\int_{B^{(N)}_d(\mu)} e^{-N\beta_N (U\circ \delta^{(N)}+o(1))}d\mu_0^{\otimes N} \nonumber \\
& = &  U(\mu) + \lim_{\delta\rightarrow 0} \liminf_{N\rightarrow \infty} -\frac{1}{N\beta_N}\log\int_{B^{(N)}_d(\mu)} d\mu_0^{\otimes N} \label{EqRed} \\
& = & U(\mu). \nonumber
\end{eqnarray}
where the second term in \eqref{EqRed} is zero by Sanov's theorem. This means $E=U$ and the theorem now follows from Proposition \ref{TheoremThmRedToZeroTempHamConv}.
\end{proof}

%%%%%%%%%%%%%%%%%%%%%%%%%%%%%%%%%%%%%%%%%%%%%%%%%%%%%

\subsection{Proof of Theorem \ref{MainThmLDP}}\label{SectProofLDP}
To use Theorem \ref{ThmRedToZeroTemp} we need to verify that the family $\{H^{(N)}\}$ is equi-continuous. We will use the following two lemmas
\begin{lemma}\label{LemmaLips}
The functions in $P(X)$ are Lipschitz with the Lipschitz constant $L=1$.
\end{lemma}
\begin{proof}
As the diameter of $X$ is 1 we get that the set 
$$\{d(\cdot,y)^2/2:y\in X\}$$ 
is Lipschitz with the Lipschitz constant $L=1$. Now, assume $\phi\in P(X)$ and $x_1,x_2\in X$. By definition
$$ \phi(x) = \sup_{y\in X} -d(x,y)^2/2 -\phi^c(y). $$
for all $x$. By compactness and continuity there is $y_1$ such that
$$ \phi(x_1) = -d(x_1,y_1)^2/2 - \phi^c(y_1). $$
We have
\begin{eqnarray}
\phi(x_2) & \geq & -d(x_2,y_1)^2/2 - \phi^c(y_1) = \phi(x_1) -(d(x_2,y_1)^2/2-d(x_1,y_1)^2/2) \nonumber \\
& \geq & \phi(x_1) - d(x_1,x_2). \nonumber
\end{eqnarray}
By interchanging the roles of $x_1$ and $x_2$ we get
$$ \phi(x_1) \geq \phi(x_2) - d(x_1,x_2) $$ 
and hence
$$ |\phi(x_1)-\phi(x_2)| \leq d(x_1,x_2). \qedhere $$
\end{proof}
 We say that a function, $\Phi$, on $\R^n$ is \emph{$\lambda$-convex} if $\Phi-\lambda\frac{x^2}{2}$ is convex. 
\begin{lemma}\label{LemmaLogConv}
Assume $\Phi_\alpha$ is a family of functions on $\R^n$ parametrized over some set $A$. Assume that for all $\alpha\in A$, $\Phi_\alpha$ is $\lambda$-convex. Let $\sigma$ be a probability measure on $A$. Then 
$$ \log \int e^{\Phi_\alpha} d\sigma(\alpha) $$
is $\lambda$-convex.
\end{lemma}
\begin{proof}
Assume first $\lambda=0$. By the convexity of $\Phi_\alpha$ in $x$ and H\"older's inequality we get
\begin{eqnarray} 
\int_A e^{\Phi_\alpha(tx_1+(1-t)x_0)} d\sigma(\alpha) & \leq & \int_A e^{t\Phi_\alpha(x_1)+(1-t)\Phi_\alpha(x_0)} d\sigma(\alpha) \nonumber \\
& \leq & \left(\int_A e^{\Phi_\alpha(x_1)}d\sigma(\alpha)\right)^t\left(\int_A e^{\Phi_\alpha(x_0)}d\sigma(\alpha)\right)^{(1-t)} \nonumber
\end{eqnarray}
and hence, taking the logarithm of both sides of this inequality,
\begin{eqnarray} & \log\int_A e^{\Phi_\alpha(tx_1+(1-t)x_0)} d\sigma(\alpha) &\nonumber \\
 \leq & t\log\int_A e^{\Phi_\alpha(x_1)}d\sigma(\alpha)+(1-t)\log\int_X e^{\Phi_\alpha(x_0)}d\sigma(\alpha). & \nonumber
\end{eqnarray}
For the general case, note that 
$$ \log\int_A e^{\Phi_\alpha(x)} d\sigma(\alpha) - \lambda\frac{x^2}{2} = \log\int_A e^{\Phi_\alpha(x)-\lambda x^2/2} d\sigma(\alpha) $$
which is convex by the case considered above. 
\end{proof}

We get
\begin{corollary}\label{CorHNLip}
The normalized energy functions 
$$\{H^{(N)}/N: k\in \N\}$$
is an equi-continuous family (in the sense of \eqref{EquiContEq}). 
\end{corollary}
\begin{proof}
We claim that 
\begin{equation} c^{(N)}_p = \frac{1}{k}\log \sum_{m\in \Z^n+p} e^{-k|x-m|^2/2} \in P(X) \label{CorHnLipEq} \end{equation}
for all $p\in X$ and $k\in \N$. To prove the claim it suffices to prove that \eqref{CorHnLipEq} is $-1$-convex. This follows from Lemma \ref{LemmaLogConv} as $-|x-m|^2/2$ is $-1$-convex for all $m\in \R^n$. Further, fixing all but one variable we get a function on $X$ given by
\begin{eqnarray} x & \mapsto & H^{(N)}(x_1,\ldots x_{i-1}, x, x_{i+1},\ldots, x_n) \nonumber \\
& = & \frac{1}{k}\log \sum_\sigma e^{-kc^{(N)}_{p_{\sigma(i)}}(x)}\prod_{j \not= i} e^{-kc^{(N)}_{p_{\sigma(j)}}(x_j)} \nonumber
\end{eqnarray}
By Lemma \ref{LemmaLogConv} this function is in $P(X)$. By Lemma \ref{LemmaLips} it satisfies the Lipschitz constant $1$. This means, if $x=(x_1,\ldots x_N)$ and $y=(y_1,\ldots, y_N)$ are points in $X^N$, that 
\begin{eqnarray} 
& |\frac{1}{N}H^{(N)}(x_1, \ldots, x_N) - \frac{1}{N}H^{(N)}(y_1, \ldots, y_N)| & \nonumber \\
\leq & \frac{1}{N}\sum_i \left|H^{(N)}(x_1, \ldots, x_{i-1}, y_i,\ldots y_N) - H^{(N)}(x_1, \ldots, x_i, y_{i+1},\ldots y_N)\right| \nonumber \\
\leq & \sum_i  d(x_i,y_i). & \label{CorHNLipEq1} 
\end{eqnarray}
As $H^{(N)}$ is symmetric we may reorder $\{x_i\}$ so that 
$$ \sum_i  d(x_i,y_i) = \inf_{\sigma} \sum_i  d(x_i,y_{\sigma(i)}) $$
and hence the right hand side of \eqref{CorHNLipEq1} equals $d^{(N)}(x,y)$. This implies $H^{(N)}/N$ is equi-continuous in the sense of \eqref{EquiContEq}.
\end{proof}

\begin{proof}[Proof of Theorem~\ref{MainThmLDP}]
By Theorem~\ref{ThmBetaInf} and Theorem~\ref{ThmRedToZeroTemp} we only need to verify that the family $\{H^{(N)}/N\}$ is uniformly bounded and equi-continuous. The latter was proved in Corollary \ref{CorHNLip}. To see that $\{H^{(N)}/N\}$ is uniformly bounded recall that in the proof of Theorem \ref{ThmBetaInf} we proved that $-\frac{1}{k}\log \Psi^{(N)}_p(x)\rightarrow d(x,p)/2$ uniformly in $x$ and $p$. Since $d(\cdot,\cdot)$ is bounded on $X\times X$ we get that there is constants $c,C\in \R$ such that, for all but finitely many $N$, 
\begin{equation} c\leq \frac{1}{k}\log \Psi^{(N)}_p(x) \leq C \label{PsiBound} \end{equation}
for all $x,p$. As the functions $\{\frac{1}{k}\log \Psi^{(N)}_p\}$ are bounded on $X$ and there is only finitely many functions for each $N$, we may choose $c$ and $C$ such that \eqref{PsiBound} holds for all $N$. We get
$$ H^{(N)}(x)/N = \frac{1}{kN} \log \sum_{\sigma} \prod_i e^{\log \Psi_{p_i}(x)} \leq \frac{1}{kN} \log \sum_{\sigma} \prod_i e^{kC} = \frac{\log N!}{kN} + C $$ 
and 
$$ H^{(N)}(x)/N = \frac{1}{kN} \log \sum_{\sigma} \prod_i e^{\log \Psi_{p_i}(x)} \geq \frac{1}{kN} \log \prod_i e^{kc} = c $$
for all $N$ and $x\in X^N$.
This proves the theorem.
\end{proof}

%%%%%%%%%%%%%%%%%%%%%%%%%%%%%%%%%%%%%%%%%%%%%%%%%%%%%%%

%%%%%%%%%%%%%%%%%%%%%%%%%%%%%%%%%%%%%%%%%%%%%%%%%%%%%%%

%%%%%%%%%%%%%%%%%%%%%%%%%%%%%%%%%%%%%%%%%%%%%%%%%%%%%%%

\section{The Rate Function and its relation to Monge Amp\`ere equations}\label{SectRate}
In this section we will show how the rate function, $G$, in Theorem \ref{MainThmLDP} is related to Monge-Amp\`ere equations. More precisely, we will establish a variational approach to equation \eqref{MAEqGen} and then show that, under a certain condition, the minimizers of the $G$ are the Monge-Amp\`ere measures of solutions to $\eqref{MAEqGen}$ (see Lemma~\ref{LemmaDual}). This will allow us to finish the proof of Theorem \ref{MainThmGen}. 

%%%%%%%%%%%%%%%%%%%%%%%%%%%%%%%%%%%%%%%%%%%%%%%%%%%%%%

\subsection{The Variational Approach to Equation \eqref{MAEqGen}}\label{SectLemmaF}
In the variational approach to equation \eqref{MAEqGen} it is convenient to consider its normalized version:
\begin{equation} \MA(\phi) = \frac{e^{\beta \phi} \mu_0}{\int_X e^{\beta \phi} d\mu_0}. \label{MAEqNorm} \end{equation}
We see that this equation is invariant under the action of $\R$ on $P(X)$ given by 
\begin{equation} C \mapsto (\phi \mapsto \phi + C). \label{EqRAction} \end{equation}
Now, we will say that an equation admits a unique solution modulo $\R$ if, for any two solutions $\phi_1,\phi_2\in C(X)$, $\phi_1-\phi_2$ is constant. It is easy to verify that \eqref{MAEqGen} admits a unique solution if and only if \eqref{MAEqNorm} admits a unique solution modulo $\R$. We will consider a certain energy functional (the analog of the Ding functional in complex geometry) whose stationary points correspond to weak solutions of \eqref{MAEqGen}. For given data $(\mu_0,\beta)$ this energy functional has the form
$$ F(\phi) = \xi(\phi) + \frac{1}{\beta}I_{\mu_0}( \beta\phi). $$
where $I_{\mu_0}$ is defined as
$$ I_{\mu_0}(\phi) = \log \int_X e^{\phi}\mu_0. $$
\begin{lemma}\label{LemmaF}
Let $\beta\not=0$. The functional $I_{\mu_0}$ is Gateaux differentiable and 
$$ dI_{\mu_0}|_{\phi} = \frac{e^{\phi}\mu_0}{\int_X e^{\phi}d\mu_0}. $$
Consequently, $F$ is Gatueux differentiable and $\phi$ is a stationary point of $F$ if and only if $\phi$ is a weak solution (in the sense of Section \ref{SectPX}) to \eqref{MAEqGen}. 
\end{lemma}
\begin{proof}
Let $v\in C(X)$. As $v$ is bounded an application of the dominated convergence theorem gives  
\begin{eqnarray} \frac{d}{dt}|_{t=0} I(\phi+tv) & = & \frac{\frac{d}{dt}|_{t=0}\int_X  e^{\phi+tv} d\mu_0}{\int_X e^{\phi} d\mu_0} \nonumber \\
& = & \frac{\int_X \frac{d}{dt}|_{t=0} e^{\phi+tv} d\mu_0}{\int_X e^{\phi} d\mu_0} \nonumber \\
& = & \frac{\int_X v e^{\phi} d\mu_0}{\int_X e^{\phi} d\mu_0}, \nonumber
\end{eqnarray}
proving the first two statements of the lemma. By Corollary~\ref{CorrDXi}, $\xi$ is differentiable and $d\xi|_\phi = -\MA(\phi)$. This means $F$ is Gateaux differentiable and 
$$ dF|_\phi = -\MA(\phi)+\frac{e^{\phi}\mu_0}{\int_X e^{\phi}d\mu_0} $$
proving the last statements of the lemma.
\end{proof}

%%%%%%%%%%%%%%%%%%%%%%%%%%%%%%%%%%%%%%%%%%%%%%%%%%%%%%

\subsection{The Minimizers of the Gibbs Free Energy}\label{SectDual}
We will use the following well know property of the relative entropy function in the proof of Lemma~\ref{LemmaDual}.
\begin{lemma}\label{LemmaI}
Let $\mu\in \M_1(X)$ and $\phi\in C(X)$. Then
\begin{equation} I_{\mu_0}(\phi) + Ent_{\mu_0}(\mu) \geq \int_X \phi d\mu \label{EntIneq} \end{equation}
with equality if and only if $\mu=dI_{\mu_0}|_\phi$. 
\end{lemma}
\begin{proof}
Assume first that $\mu$ is absolutely continuous with respect to $\mu_0$ and $\mu_0$ is absolutely continuous with respect to $\mu$. By Jensen's inequality
\begin{eqnarray} I_{\mu_0}(\phi) & = & \log \int_X e^{\phi} \frac{\mu_0}{\mu} d\mu \nonumber \\
& \geq & \int_X \phi d\mu - \int_X \log\frac{\mu}{\mu_0}d\mu \nonumber \\
& = & \int_X \phi d\mu - Ent_{\mu_0}(\mu) \nonumber
\end{eqnarray}
with equality if and only if $e^{\phi} \frac{\mu_0}{\mu}$ is constant, or, equivalently, $\mu$ is proportional to $e^{\phi}\mu_0$. As $\mu$ is a probability measure this means 
$$ \mu = \frac{e^\phi \mu_0}{\int_X e^\phi d\mu_0} = dI|_\phi $$ 
proving the lemma in this special case. 
If $\mu$ is not absolutely continuous with respect to $\mu_0$ then $Ent_{\mu_0}(\mu)=+\infty$ and the equality holds trivially. Finally, when $\mu$ is absolutely continuous with respect to $\mu_0$ but $\mu_0$ is not absolutely continuous with respect to $\mu$, then replacing $\mu_0$ by $\chi\mu_0$, where $\chi$ is the characteristic function of the support of $\mu$ doesn't change the right hand side of \eqref{EntIneq}. Since 
$$ I_{\mu_0}(\phi) \geq \log\int e^{\phi} \chi d\mu_0 $$
this reduces this case to the case when $\mu_0$ is absolutely continuous with respect to $\mu$.
\end{proof}
We can now prove Lemma~\ref{LemmaDual}.
\begin{lemma}\label{LemmaDual}
Assume $\beta\not= 0$, $F$ admits a unique minimizer modulo $\R$ and $\phi_*$ is a minimizer of $F$. Then
\begin{equation} \mu_* = \MA(\phi_*) \label{MuMin} \end{equation}
is the unique minimizer of the rate function
$$ G(\mu) = \beta W^2(\mu,dx) + Ent_{\mu_0}(\mu) + C_{\mu_0,\beta} $$
defined in Theorem \ref{MainThmLDP}. 
\end{lemma}
\begin{remark}
Note that $\phi_1-\phi_2=C$ implies $\phi^c_1-\phi_2^c=-C$ and hence 
$$\MA(\phi_1)=(\nabla^c\phi_1^c)_* dx = (\nabla^c \phi_2^c)_*dx = \MA(\phi_2).$$
This means that, under the assumptions of Lemma~\ref{LemmaDual}, $\mu_*$ is uniquely determined by \eqref{MuMin}.
\end{remark}
\begin{proof}[Proof of Theorem \ref{LemmaDual}]
Note that by Corollary \ref{CorrDXi} and Lemma \ref{LemmaI} we have, for all $\mu\in M_1(X)$ and $\phi\in C(X)$, the two inequalities
\begin{eqnarray} 
W^2(\mu,dx)+\xi(\phi) & \geq & -\int \phi d\mu \label{LemmaDualEq1}\\ 
Ent_{\mu_0}(\mu) + I_{\mu_0}(\phi) & \geq & \int \phi d\mu \label{LemmaDualEq2}
\end{eqnarray}
where equality in \eqref{LemmaDualEq1} is characterized by 
\begin{equation} d\xi|_\phi = -MA(\phi) = -\mu \label{EqInWXi} \end{equation}
and equality in \eqref{LemmaDualEq2} is characterized by $dI|_\phi = \mu.$
We will start with the case $\beta>0$. Let $\mu\in \M_1(X)$ and $\phi_*$ be the minimizer of $F$. Applying \eqref{LemmaDualEq1} to the pair $\mu$ and $\phi_*$ and \eqref{LemmaDualEq2} to the pair $\mu$ and $\beta\phi_*$ we get
\begin{eqnarray}
G(\mu) & = & \beta W^2(\mu,dx) + Ent(\mu) \nonumber \\
& \geq & -\beta\int \phi_* d\mu -\beta\xi(\phi_*) + \int \beta \phi_* d\mu - I(\beta\phi_*) \label{LemmaDualEq4} \nonumber \\
& = & -\beta\left(\xi(\phi_*) + \frac{1}{\beta}I(\beta\phi_*)\right) = -\beta F(\phi_*) \nonumber
\end{eqnarray}
with equality if and only if $d\xi|_{\phi_*} = -MA(\phi_*) = -\mu$ and $\mu=dI|_{\phi_*}$ which, since $d\xi|_{\phi_*} + dI|_{\phi_*} = 0$, is true if and only if $\mu=\MA(\phi_*)$.  
For the case $\beta<0$, let $\mu\in\M_1(X)$. By Corollary~\ref{CorrDXi} we may take $\phi$ to satisfy equality in \eqref{LemmaDualEq1} and hence \eqref{EqInWXi}. A similar application of \eqref{LemmaDualEq1} and \eqref{LemmaDualEq2} as above, keeping in mind that we have equality in \eqref{LemmaDualEq1}, give
\begin{eqnarray}
G(\mu) & = & \beta W^2(\mu,dx) + Ent(\mu) \nonumber \\
& \geq & -\beta\int \phi d\mu -\beta\xi(\phi) + \int \beta \phi d\mu - I(\beta\phi) \label{LemmaDualEq6} \\
& = & -\beta\left(\xi(\phi) + \frac{1}{\beta}I(\beta\phi)\right) = -\beta F(\phi) \geq -\beta F(\phi_*). \label{LemmaDualEq5}
\end{eqnarray}
Moreover, equality in \eqref{LemmaDualEq5} holds if and only if $\phi=\phi_*$. But that means $d I|_\phi = -d\xi|_\phi = \mu$, hence we have equality in \eqref{LemmaDualEq6} as well. This implies $G(\mu)\geq-\beta F(\phi_*)$ with equality if and only if $\mu=\MA(\phi_*)$. 
\end{proof}

\subsection{Proof of Theorem \ref{MainThmGen} and Corollary \ref{CorrGen}}
\begin{proof}[Proof of Theorem \ref{MainThmGen}]
Let $\phi_*$ be the unique solution to \eqref{MAEqGen}. It follows that \eqref{MAEqNorm} admits a unique solution modulo $\R$ and that $\phi_*$ is a solution to \eqref{MAEqNorm}. Now, we will use two results from the next chapter. Namely that any stationary point of $F$ is a \emph{smooth} solution to \eqref{MAEqNorm} (see Section \ref{SectRegularity}) and that $F$ always admit a minimizer (see Section \ref{SectExistence}). Under our assumptions, this implies $F$ admits a unique minimizer modulo $\R$ and that $\phi_*$ is a minimizer of $F$. Using Lemma \ref{LemmaDual} we get that $G$ admits the unique minimizer $\mu_*$ satisfying $\mu_*=\MA(\phi_*)$. 

We want to prove that $\Gamma^{(N)}\rightarrow \delta_{\mu_*}$ in the weak* topology on $\M_1(\M_1(X))$. By the Portmanteau Theorem it suffices to verify that
\begin{equation} \limsup_{N\rightarrow \infty} \Gamma^{(N)}(F) \leq \delta_{\mu_*}(F)  \label{MainThmGenEq1} \end{equation}
for all closed $F\subset \M_1(X)$. If $\mu_*\in F$ then \eqref{MainThmGenEq1} holds trivially. Assume $\mu_*\notin F$. Recall that $\M_1(X)$ is compact. This means the closed subset $F$ is compact. Since $G$ is lower semi-continuous there is $\mu_F\in F$ such that $\inf_F G = G(\mu_F)$. As $\mu_*\notin F$ is the unique point where $G = \inf G = 0$ we get that $G(\mu_F)=\inf_F G >0$. By the large deviation principle in Theorem \ref{MainThmLDP} 
$$ \limsup_{N\rightarrow \infty} \frac{1}{r_N} \log \Gamma^{(N)}(F) \leq -\inf_F G < 0. $$
As $r_N\rightarrow \infty$ we get that $\limsup  \log \Gamma^{(N)}(F) = -\infty$ and $\limsup \Gamma^{(N)}(F)=0$. This proves the theorem. 
\end{proof}

\begin{proof}[Proof of Corollary \ref{CorrGen}]
Equation \eqref{EqMainThmGen} implies the first marginals of $\mu_\beta^{(N)}$,
$$ \int_{X^{N-1}}\mu^{(N)}_\beta, $$
converges to $\mu_*$ in the weak* topology of $M_1(X)$ (see Proposition 2.2 in \cite{Sznitman}). Now, $e^{\beta \phi_N}$ is the density with respect to $\mu_0$ of the first marginal of $\mu_\beta^{(N)}$. 

We claim that the collection $\{\phi^{(N)}: k\in \N\}$ is equi-continuous and uniformly bounded. To see this, note that by Lemma \ref{LemmaLogConv}, $\phi^{(N)}$ is $-1$-convex and hence in $P(X)$. By Lemma \ref{LemmaLips} the functions $\{\phi^{(N)}, k\in \N\}$ satisfy the Lipschitz constant $L=1$. As
$$ \int_X e^{\beta\phi_N} \mu_0 = \int_{X^N} \mu_\beta^{(N)} = 1 $$
for all $N$, this means there are constants $c,C\in \R$, independent of $N$, such that $c\leq \phi_N\leq C$. This proves the claim. By the Arzel\`a-Ascoli theorem there is some function $\phi_\infty\in C(X)$ such that 
$$\phi_N\rightarrow \phi_\infty$$ 
uniformly. As 
$$ e^{\beta \phi_N}\mu_0 =  \int_{X^{N-1}}\mu^{(N)}_\beta \rightarrow \mu_* = e^{\beta\phi_*}\mu_0 $$
in the weak* topology of $M_1(X)$ we get that $\phi_\infty = \phi_*$ almost everywhere with respect to $\mu_0$. As $\mu_0$ has full support and $\phi_\infty,\phi\in C(X)$, this means $\phi_\infty = \phi_*$.
\end{proof}

%%%%%%%%%%%%%%%%%%%%%%%%%%%%%%%%%%%%%%%%%%%%%%%%%%%%%%

%%%%%%%%%%%%%%%%%%%%%%%%%%%%%%%%%%%%%%%%%%%%%%%%%%%%%%

%%%%%%%%%%%%%%%%%%%%%%%%%%%%%%%%%%%%%%%%%%%%%%%%%%%%%%

\section{Existence and Uniqueness of Solutions}\label{SectGibbsEnergy}
In this section we will treat questions of existence and uniqueness of solutions to \eqref{MAEqGen} for different data $(\mu_0,\beta)$. First of all we will prove that, for any data $(\mu_0,\beta\not=0)$, \eqref{MAEqGen} admit a weak solution. We will then explain how to reduce the problem of regularity to the case considered in \cite{BermanBerndtsson}, where the authors use Caffarelli's interior regularity theory for Monge-Amp\`ere equations. In the last part of the section we treat uniqueness. We first prove the claim made in Remark~\ref{RemBetaPos}, namely that as long as $\beta>0$ equation \eqref{MAEqGen} admits at most one solution. Finally we prove Theorem \ref{ThmFUniq} regarding $\beta\in [-1,0)$ and $\mu_0=\gamma$.

%%%%%%%%%%%%%%%%%%%%%%%%%%%%%%%%%%%%%%%5

\subsection{Existence of Weak Solutions}\label{SectExistence}
First of all, Lemma \ref{LemmaLips} implies $P(X)$ satisfies the following (relative) compactness property:
\begin{lemma}\label{LemmaPComp}
Let $\{\phi_k\}$ be a sequence of functions in $P(X)$ such that $\inf_X \phi_k = 0$ for all $k$, then there is $\phi\in C(X)$ such that, after passing to a subsequence, $\phi_k\rightarrow \phi$ uniformly. 
\end{lemma}
\begin{proof}
By lemma \ref{LemmaLips}, $\{\phi_k\}$ are Lipschitz with a uniform Lipschitz constant. As $X$ has finite diameter and $\inf_X \phi_k = 0$ for all $k$ this means $\{\phi_k\}$ is also uniformly bounded, hence the lemma follows from the Arzel\`a-Ascoli theorem.
\end{proof}
\begin{lemma}\label{LemmaFP}
Let $\phi\in C(X)$ and 
$$ F(\phi) = \xi(\phi) + \frac{1}{\beta}I_{\mu_0}( \beta\phi). $$ 
Then 
\begin{equation} F\left((\phi^c)^c\right) \leq F(\phi). \label{EqFP1} \end{equation}
Moreover, if $\mu_0$ has full support, then equality holds in \eqref{EqFP1} if and only if $\phi\in P(X)$. 
\end{lemma}
\begin{proof}
Recall that $\phi^c\in P(X)$, and hence $((\phi^c)^c)^c = \phi^c$ for all $\phi\in C(X)$. Also, $(\phi^c)^c\leq \phi$ for all $\phi\in C(X)$. This means $\xi(\phi)=\xi((\phi^c)^c)$ and
\begin{equation} I_{\mu_0}((\phi^c)^c) = \frac{1}{\beta}\log\int_X e^{\beta(\phi^c)^c} d\mu_0 \leq \frac{1}{\beta}\log\int_X e^{\beta\phi} d\mu_0 = I_{\mu_0}(\phi). \label{EqFP2} \end{equation}
and hence
\begin{equation} F\left((\phi^c\right)^c)\leq F(\phi). \label{EqFP3} \end{equation}
Assume $\mu_0$ has full support. Then, if $\phi\notin P(X)$ and hence $(\phi^c)^c(x)< \phi(x)$ for some $x\in X$, then, as both $(\phi^c)^c$ and $\phi$ are continuous and $\mu_0$ has full support, strict inequality holds in \eqref{EqFP2} and \eqref{EqFP3}. This proves the lemma.
\end{proof}

\begin{lemma}\label{LemmaExistence}
Let $\beta\in \R\setminus\{0\}$. Then $F$ admits a minimizer. In other words, \eqref{MAEqGen} admits a weak solution. 
\end{lemma}
\begin{proof}
Recall that 
$$ F(\phi) = \xi(\phi) + \frac{1}{\beta}I(\beta\phi). $$
By the Dominated Convergence Theorem $\frac{1}{\beta}I(\beta\phi)$ is continuous in $\phi$. By Lemma~\ref{LemmaXiCont}, $\xi$ is continuous. This means $F$ is continuous. Let $\phi_k$ be a sequence such that $F(\phi_k)\rightarrow \inf F$. By Lemma \ref{LemmaFP} we may assume $\phi_k\in P(X)$ for all $k$. As $F$ is invariant under the action of $\R$ given in \eqref{EqRAction} we may assume $\phi_k$ satisfies $\inf \phi_k=0$ for all $k$. By Lemma \ref{LemmaPComp}, after possibly passing to a subsequence, $\phi_k\rightarrow \phi$ for some $\phi\in C(X)$. By continuity $F(\phi) = \lim_{k\rightarrow \infty}F(\phi_k) = \inf F$, hence $\phi$ is a minimizer of $F$. 
\end{proof}

%%%%%%%%%%%%%%%%%%%%%%%%%%%%%%%%%%%%%%%%%%%%%%%%%%%%%%

\subsection{Regularity}\label{SectRegularity}
In a numbers of papers (see \cite{Caffarelli1}, \cite{Caffarelli2}, \cite{Caffarelli4}) Caffarelli developed a regularity theory for various types of weak solutions to Monge-Amp\`ere equations. In particular, Caffarelli's theory applies to so called \emph{Alexandrov solutions}. Recall that if $f$ is a smooth function on $\R^n$, then a convex function $\Phi$ on $\R^n$ is an Alexandrov solution to the equation 
$$ \det(\Phi_{ij}) = f $$
if, for any borel measurable $E\subset \Omega$, 
$$ \int_E f dx = \int_{\partial \Phi(E)} dx $$
where $\partial\Phi(E)$ is the image of $E$ under the multivalued gradient mapping, in other words
$$ \partial\Phi(E) = \{ y\in \R^n: \Phi(x) + \Phi^*(y) = \bracket{x}{y} \textnormal{ for some } x\in E\}. $$
We have the following lemma:
%Let $\Omega$ be a convex, bounded and open subset of $\R^n$, $f$ a function on $\Omega$ and $\phi$ a function on $\overline\Omega$ such that $\phi=0$ on $\partial\Omega$ and
%\begin{equation} \det(\Phi_{ij}) = f \label{EqCaff} \end{equation}
%on $\Omega$.  
%Recall that $\Phi$ is a so called \emph{Alexandrov solution} to \eqref{EqCaff} if, for any borel measurable $E\subset \Omega$, 
%$$ \int_E f dx = \int_{\partial \Phi(E)} dx $$
%where $\partial\Phi(E)$ is the image of $E$ under the multivalued gradient mapping, in other words
%$$ \partial\Phi(E) = \{ y\in \R^n: \Phi(x) + \Phi^*(y) = \bracket{x}{y} \textnormal{ for some } x\in E\}. $$
%In particular, Caffarelli's theory applies to Alexandrov solutions. An important observation is that as long as $\Phi$ is proper we may take a sequence of constants $C_i$ such that $C_i\rightarrow \infty$ and consider the function $\Phi = \Phi-C_i$ on $\Omega = \{\Phi \leq C_i\}$ to get the case considered by Caffarelli. We have the following lemma:
\begin{lemma}\label{LemmaEqLift}
Assume $\mu_0$ is absolutely continuous with density $f$ with respect to dx, $\beta\in \R$ and
\begin{equation} \MA(\phi) = e^{\beta\phi} \mu_0. \label{RegMAEq} \end{equation}
in the sense of Definition~\ref{DefWeakOp}. Then $\Phi = \phi\circ \pi + x^2/2$ is an Alexandrov solution to the equation
\begin{equation} \det(\Phi_{ij}) = e^{\beta(\Phi-x^2/2)} f\circ\pi \label{MAEqGenLift} \end{equation} 
on $\R^n$. Moreover, $\Phi$ is proper.
\end{lemma}
\begin{proof}
Assume $E$ is a Borel measurable subset of $\R^n$. To prove the first point in the lemma we need to prove 
$$ \int_E e^{\beta(\Phi-x^2/2)}f\circ \pi dx = \int_{\partial\Phi(E)} dx. $$
Let $C_0=[0,1)^n\subset \R^n$ and $\{C_i\}$ be a collection of disjoint translates of $C_0$ such that $E\subset \cup C_i$. Let $E_i=E\cap C_i$. We have
$$ \int_E e^{\beta(\Phi-x^2/2)}f\circ \pi dx = \sum_i \int_{E_i} e^{\beta(\Phi-x^2/2)}f\circ \pi dx = \sum_i \int_{\pi(E_i)} e^{\beta\phi} f dx  $$
and by \eqref{RegMAEq}
$$ \sum_i \int_{\pi(E_i)} e^{\beta\phi} f dx = \sum_i \int_{(\nabla^c\phi^c)^{-1}(\pi(E_i))} dx. $$

Now, we claim that $\pi$ maps $(\nabla\Phi^*)^{-1}(E_i)$ bijectively onto 
$$(\nabla^c\phi^c)^{-1}(\pi(E_i))$$ 
for all $i$. To see this note that if $y\in (\nabla\Phi^*)^{-1}(E_i)$, then 
$$\nabla^c\phi^c\circ \pi(y) = \pi \circ\nabla\Phi^*(y) \in \pi(E_i), $$ 
hence $\pi(y)\in (\nabla^c\phi^c)^{-1}(\pi(E_i))$. On the other hand, if $y\in (\nabla^c\phi^c)^{-1}(\pi(E_i))$, let $\tilde x$ be the unique lift of $\nabla^c\phi^c(y)$ in $E_i$. Moreover, let $\tilde y$ be a lift of $y$ in $\R^n$. Since $\nabla^c\phi^c(y)=x$ we have $\nabla\Phi^*(\tilde y) = \tilde x + m_0$ for some $m_0\in \Z^n$. We have that 
$$\pi^{-1}(y)=\{\tilde y+m: m\in \Z^n\}$$ 
and by \eqref{DefCvx} 
$$ \nabla\Phi^*(\tilde y+m) = \nabla\Phi^*(\tilde y)+m=\tilde x + m_0+m. $$
We conclude that $\nabla\Phi^*(\tilde y+m)\in E_i$ if and only if $m=-m_0$ and then $\nabla\Phi^*(\tilde y+m)=\tilde x$. This means $\pi$ maps $(\nabla\Phi^*)^{-1}(E_i)$ bijectively onto $(\nabla^c\phi^c)^{-1}(\pi(E_i))$ as claimed. We get
$$ \sum_i \int_{(\nabla^c\phi^c)^{-1}(\pi(E_i))} dx = \sum_i\int_{(\nabla\Phi^*)^{-1}(E_i)} dx = \int_{(\nabla\Phi^*)^{-1}(E)} dx $$
where the second inequality holds since the sets $(\nabla\Phi^*)^{-1}(E_i)$ are disjoint. Now, let $\dom\nabla\Phi^*$ be the set where $\nabla\Phi^*$ is defined. We have 
\begin{eqnarray} \dom\nabla\Phi^*\cap \partial \Phi(E) & = & \{y\in \R^n: \nabla\Phi^*(y) = x \textnormal{ for some } x\in E\} \nonumber \\
&  = & (\nabla\Phi^*)^{-1}(E). \nonumber \end{eqnarray}
Since $\Omega\setminus \dom\nabla\Phi^*$ is a zero-set with respect to $dx$ we have
$$ \int_{(\nabla\Phi^*)^{-1}(E)} dx = \int_{\partial\Phi(E)} dx $$
which proves the first part of the lemma.

%Now, we claim $\pi$ is injective on $(\nabla\Phi^*)^{-1}(C_i)$ for all $i$. To see this, assume $y_1,y_2\in (\nabla \Phi^*)^{-1}(C_i)$ for some $i$ and $\pi(y_1)=\pi(y_1)$. This means
%$$ \pi\circ\nabla\Phi^*(y_1)=\nabla^c\phi^c\circ\pi(y_1)=\nabla^c\phi^c\circ \pi(y_2) = \pi\circ\nabla\Phi^*(y_2) $$
%and since $\pi$ is injective on $E_i$
%$$ \nabla\Phi^*(y_1)=\nabla\Phi^*(y_2). $$
%Now, $\pi(y_1)=\pi(y_2)$ implies $y_1=y_2+m$ for some $m\in \Z^n$ and by \eqref{DefCvx}
%$$ \nabla\Phi^*(y_1)=\nabla \Phi^*(y_2+m)=\nabla\Phi(y_2)+m, $$
%hence $m=0$ and $y_1=y_2$. This proves the claim.  
%Moreover, let $\dom\nabla\Phi^*$ be the set where $\nabla\Phi^*$ is defined. We have 
%$$ \dom\nabla\Phi^*\cap \partial \Phi(E) = \{y\in \R^n: \nabla\Phi^*(y) = x \textnormal{ for some } x\in E\} = (\nabla\Phi^*)^{-1}(E). $$
%Since $\Omega\setminus \dom\nabla\Phi^*$ is a zero-set with respect to $dx$ we have
%\begin{eqnarray} 
%\int_E e^{\beta(\Phi-x^2/2)}f\circ \pi dx & = & \sum_i \int_{E_i} e^{\beta(\Phi-x^2/2)}f\circ \pi dx = \sum_i \int_{\pi(E_i)} e^{\beta\phi} f dx \nonumber \\
%& = & \sum_i \int_{(\nabla^c\phi^c)^{-1}(\pi(E_i))} dx = \sum_i\int_{\pi((\nabla\Phi^*)^{-1}(E_i)} dx \nonumber \\
%& = & \sum_i\int_{(\nabla\Phi^*)^{-1}(E_i)} dx = \int_{(\nabla\Phi^*)^{-1}(E)} dx  \nonumber \\
%& = & \int_{\partial\Phi(E)} dx \nonumber
%\end{eqnarray}
%which proves the first part of the lemma. 

To see that $\Phi$ is proper, note that since $\phi$ is continuous it is bounded on $X$. Let $C=\inf_X\phi$. We get
$$ \Phi(x) = \phi(\pi x) + \frac{x^2}{2} \geq C-1+|x|. \qedhere $$
\end{proof}

\begin{lemma}\label{LemmaRegularity}
Assume $\mu_0$ is absolutely continuous with smooth density with respect to $dx$ and $\phi\in P(X)$ satisfies \eqref{MAEqGen} in the sense of Definition \ref{DefWeakOp}. Then $\phi$ is smooth. 
\end{lemma}
\begin{proof}
We refer to \cite{BermanBerndtsson} (more precisely, step three in the proof of Theorem 1.1) where the authors explain why, by Caffarelli's regularity theory, proper Alexandrov solutions on $\R^n$ to the equation 
\begin{equation} \det(\Phi_{ij}) = F(\Phi,x), \label{EqBoRob} \end{equation}
where $F$ is smooth, are smooth. Strictly speaking the authors use an additional assumption of ''finite energy'', but the only way this is used is to guarantee properness of $\Phi$. By Lemma \ref{LemmaEqLift}, $\Phi=\phi\circ \pi + x^2/2$ is proper and satisfies \eqref{MAEqGenLift} in the Alexandrov sense. As \eqref{MAEqGenLift} is indeed a special case of \eqref{EqBoRob} this proves the lemma.
\end{proof}

%%%%%%%%%%%%%%%%%%%%%%%%%%%%%%%%%%%%%%%%%%%%%%%%%%%%%%

\subsection{Uniqueness}\label{SectStrictConv}
We first prove the claim made in Remark~\ref{RemBetaPos}.
\begin{theorem}\label{ThmUniqPos}
Let $\mu_0\in \M_1(X)$ be absolutely continuous with smooth density with respect to $dx$ and $\beta>0$. Then \eqref{MAEqGen} admits a unique solution. 
\end{theorem} 
\begin{proof}
By Lemma \ref{LemmaExistence} and Lemma \ref{LemmaRegularity} there always exist a solution to $\eqref{MAEqGen}$. To prove uniqueness it suffices to prove that the normalized equation \eqref{MAEqNorm} admits a unique solution modulo $\R$, in other words that $F$ admits a unique minimizer modulo $\R$. Assume then $\phi_0$ and $\phi_1$ satisfies
\begin{equation} F(\phi_0) = F(\phi_1) = \inf_{C(X)} F. \label{FInf} \end{equation}
Let $\phi_t=t\phi_1+(1-t)\phi_0$. Applying Lemma~\ref{LemmaLogConv} with $A=X$ and $\Phi_\alpha(x)=\phi_x(\alpha)$ gives that 
$$ I_{\mu_0}(\phi_t) = \log\int_X e^{\phi_t} d\mu_0 $$
is convex in $t$. Now, $\xi(\phi_t)$ is convex in $t$ by Lemma \ref{XiConv}. 
%We claim that
%\begin{equation} I_{\mu_0}(\phi) = \sup_{\mu\in \M(X)} \int_X \phi d\mu - Ent_{\mu_0}(\mu) \label{Eq2UniqPos} \end{equation}
%and hence that $I_{\mu_0}$ is convex. To see this, note that by Lemma \ref{LemmaI}
%the right hand side of \eqref{Eq2UniqPos} is smaller than $I_{\mu_0}(\phi)$. Moreover, by Lemma \ref{LemmaI} again, putting $\mu=e^{\phi}d\mu_0$ gives
%$$ I_{\mu_0}(\phi) = \int_X \phi d\mu - Ent_{\mu_0}(\mu_2) $$
%hence \eqref{Eq2UniqPos} holds. 
This means $F(\phi_t)$ is convex and hence, by \eqref{FInf}, constant in $t$. It follows that $I_{\mu_0}(\phi_t)$ is affine in $t$. However, if we let $v=\frac{d}{dt}\phi_t = \phi_1-\phi_0$, then
\begin{eqnarray} \frac{d^2}{dt^2} I_{\mu_0}(\phi_t) & = & \frac{d}{dt}\left(\frac{\int_X v e^{\phi_t} d\mu_0}{\int_X e^{\phi_t}d\mu_0}\right) \nonumber \\
& = & \frac{\int_X v^2 e^{\phi_t}d\mu_0\int_X e^{\phi_t}d\mu_0 - \left(\int_X v e^{\phi_t}d\mu_0\right)^2}{\left(\int_X e^{\phi_t}d\mu_0\right)^2} \label{EqDDI}
\end{eqnarray}
%If we let $\hat v$ be the constant
%$$ \hat v = \int_X v \frac{e^{\phi_t}d\mu_0}{\int_X e^{\phi_t}d\mu_0} $$
%we get 
%\begin{eqnarray}
%\eqref{EqDDI} & = &  \int_X v^2 \frac{e^{\phi_t}d\mu_0}{\int_X e^{\phi_t}d\mu_0} - 2\hat v \int_X v \frac{e^{\phi_t}d\mu_0}{\int_X e^{\phi_t}d\mu_0} + \int_X \hat v^2 \frac{e^{\phi_t}d\mu_0}{\int_X e^{\phi_t}d\mu_0} \nonumber \\
%& = & \int_X (v-\hat v)^2 \frac{e^{\phi_t}d\mu_0}{\int_X e^{\phi_t}d\mu_0}. \nonumber
%\end{eqnarray}
%\begin{equation} \frac{d^2}{dt^2} I_{\mu_0}(\phi_t) = \frac{d}{dt}\int_X v \nu_t = \int_X v^2 \nu_t - \left(\int_X v\nu_t\right)^2 \label{EqDDI}
%\end{equation}
Further, if we let $\nu_t$ be the probability measure
$$\nu_t = \frac{e^{\phi_t}d\mu_0}{\int_X e^{\phi_t}d\mu_0}$$
and $\hat v$ be the constant
$$ \hat v = \int_X v \nu_t $$
then
$$ \eqref{EqDDI} = \int_X v^2 \nu_t - \hat v^2 =  \int_X v^2 \nu_t - 2\hat v \int_X v \nu_t + \hat v^2 = \int_X (v-\hat v)^2 \nu_t. $$
In particular, since $I_{\mu_0}(\phi_t)$ is affine in $t$ we get that $v=\hat v$, hence that $\phi_1-\phi_0$ is constant. This proves the theorem.
%By general properties of subgradients this means there is $\mu\in \M(X)$ such that 
%$$ Ent_{\mu_0}(\mu) + I_{\mu_0}(\phi_t) = \int_X\phi_t d\mu $$
%for all $t\in [0,1]$. But, by Lemma \ref{LemmaI} this only holds if 
%$$\mu=dI_{\mu_0}|_{\phi_t} =\frac{e^{\phi_t}\mu_0}{\int_X e^{\phi_t}d\mu_0} $$
%for all $t\in [0,1]$. This implies $\phi_0 = \log\mu/\mu_0 + C_0$ and $\phi_1 = \log\mu/\mu_0 +C_1$ for some $C_1,C_2\in \R$ which proves the theorem.
\end{proof}

We now turn to the proof of Theorem~\ref{ThmFUniq}. We will use
\begin{theorem}[The Prekopa Inequality \cite{Borell}, \cite{Dubuc},\cite{Prekopa}]
Let $\phi:[0,1]\times \R^n \rightarrow \R$ be a convex function. Define
$$ \hat\phi(t) = -\log \int_{\R^n} e^{-\phi(t,x)} dx. $$
Then, for all $t\in \R$ 
$$ \hat\phi(t) \leq t\hat\phi(1) + (1-t)\hat\phi(0) $$
with equality if and only if there is $v\in \R^n$ and $C\in \R$ such that
$$ \phi(t,x) = \phi(0,x-tv)+tC. $$
\end{theorem}
\begin{proof}[Proof of Theorem \ref{ThmFUniq}]
By Lemma \ref{LemmaExistence} and Lemma \ref{LemmaRegularity} there always exist a solution to $\eqref{MAEqSpec}$. Similarily as in the proof of Theorem \ref{ThmUniqPos}, to prove uniqueness it suffices to prove that $F$ admits a unique minimizer modulo $\R$.  Assume $\phi_0$ and $\phi_1$ satisfies
$$ F(\phi_0) = F(\phi_1) = \inf_{C(X)} F. $$
By Lemma \ref{LemmaFP} any minimizer of $F$ is in $P(X)$, hence $(\phi_0^c)^c = \phi_0$ and $(\phi_1^c)^c = \phi_1$. This means the following equation defines a curve in $C(X)$ connecting $\phi_0$ and $\phi_1$:
\begin{equation} \phi_t = \left( t(\phi_1)^c + (1-t)(\phi_0)^c \right)^c. \label{EqGeodesics} \end{equation}
Note that, as $P(X)$ is convex and $\phi^c_0,\phi^c_1\in P(X)$ we get $t\phi_1^c+(1-t)\phi_0^c\in P(X)$ and
\begin{equation} F(\phi_t) = \int_X t\phi_1^c+(1-t)\phi_0^c dx + \frac{1}{\beta} \log \int_X e^{\beta \phi_t} d\gamma.  \end{equation}
The first term of this is affine in $t$. The second term is given by
\begin{equation} \frac{1}{\beta} \log \int_X e^{\beta \phi_t} \sum_{m\in \Z^n} e^{|x-m|^2/2} dx = \frac{1}{\beta} \log \int_{\R^n} e^{\beta \phi_t\circ \pi-x^2/2} dx. \label{EqF} \end{equation}
Let $\Phi_t = \phi_t\circ\pi + x^2/2$. By Lemma~\ref{LemmaLegCom}, since $\phi_t$ is the $c$-transform $t\phi_1^c+(1-t)\phi_0^c$, we have
\begin{equation} \Phi_t(x) = \sup_{y\in \R^n} \bracket{x}{y} - \left(t\phi_1^c+(1-t)\phi_0^c\right)\circ\pi(y)-\frac{y^2}{2}. \label{EqProof2} \end{equation}
As 
$$\bracket{x}{y} - \left(t\phi_1^c+(1-t)\phi_0^c\right)\circ\pi(y)-\frac{y^2}{2}$$
is affine in $(t,x)$ we get that \eqref{EqProof2} is convex in $(t,x)$. It follows that, as long as $\beta\in [-1,0)$, the exponent in \eqref{EqF},
\begin{eqnarray} \beta \phi_t\circ \pi(x)-x^2/2 & = & \beta (\phi_t\circ\pi(x) + x^2/2) - (\beta+1) x^2/2 \nonumber \\
& = & \beta\Phi_t(x) - (\beta+1)x^2/2 \nonumber
\end{eqnarray}
is concave in $(t,x)$. We may then apply the Prekopa inequality to deduce that $\eqref{EqF}$ and hence $F(\phi_t)$ is convex in $t$. In particular, as $\phi_0$ and $\phi_1$ are minimizers of $F$, this means $F(\phi_t)=F(\phi_0)=F(\phi_1)$ for all $t\in [0,1]$. This imples \eqref{EqF} is affine in $t$. By the equality case in the Prekopa inequality 
$$ \beta\phi_1\circ\pi(x)-x^2/2 = \beta\phi_0\circ\pi(x-v)-(x-v)^2/2 +C $$ 
for some $C\in \R$ and $v\in \R^n$. By noting that $\phi_1\circ\pi$ and $\phi_0\circ\pi(\cdot-v)$, and hence 
$$ \beta\phi_1\circ\pi-\beta\phi_0\circ\pi(\cdot-v) = \bracket{\cdot}{v} +v^2/2 + C, $$
should descend to a function on $X$ (in other words, they should be invariant under the action of $\Z^n$), we get that $v=0$. This means $\phi_1=\phi_0+C$ which proves Theorem \ref{ThmFUniq}. 
\end{proof}

%%%%%%%%%%%%%%%%%%%%%%%%%%%%%%%%%%%%%%%%%%

%%%%%%%%%%%%%%%%%%%%%%%%%%%%%%%%%%%%%%%%%%

%%%%%%%%%%%%%%%%%%%%%%%%%%%%%%%%%%%%%%%%%%

\section{Geometric Motivation}\label{SectGeo}
The original motivation for this project comes from the paper on statistical mechanics and birational geometry by Berman \cite{BermanBiRat}. Berman introduces a thermodynamic approach to produce solutions to the complex Monge-Amp\`ere equation
\begin{equation} \MA_\C(u) = e^{\beta u} \mu_0 \label{CxMAEq} \end{equation}
on a compact K\"ahler manifold $M$. The Monge-Amp\`ere operator in \eqref{CxMAEq} is defined as
\begin{equation} (i\partial\bar\partial u + \omega_0)^n \label{CxMAOp} \end{equation}
where $n$ is the complex dimension of $M$ and $\omega_0$ is a fixed K\"ahler-form on $M$ representing the Chern class of a line bundle $L$ over $M$. A solution, $u$, should be a real valued twice differentiable function on $M$ satisfying $i\partial\bar\partial u + \omega_0>0$. As Berman's thermodynamic approach to this equation has served as an inspiration for us, we outline it here. 

The metric, $\omega_0$ determines, up to a constant, a metric  on $L$. For each $k>0$, let $N=N_k=H^0(M,L)$. By assumption on $\omega_0$, $L$ is ample and hence $N_k\rightarrow \infty$ as $k\rightarrow \infty$. Let $s_1,\ldots s_N$ be a basis of $H^0(M,L)$. Locally we may identify this basis with a collection of functions $f_1,\ldots f_N$. The map
$$ (x_1,\ldots,x_N) \mapsto \det(f_i(x_j)) $$ 
determines a section, $\det(s_1,\ldots,s_N)$, of the induced line bundle $L^{\boxtimes N_k}$ over $M_N$. The metric on $L$ induces a metric, $\norm{\cdot}$, on this line bundle and
\begin{equation} \norm{\det(s_1,\ldots,s_N)}^{2\beta/k} \mu_0 \label{CxPP} \end{equation}
determines a symmetric measure on $M^N$. Note that changing the basis of $H^0(M,L)$ will give the same result up to a multiplicative constant. As long as this measure has finite volume we may normalize it to get a symmetric probability measure on $M^N$. 

Now, Berman shows that if $\beta>0$ and the singularities of $\mu_{\C}$ are controlled in a certain way, then the point processes defined by \eqref{CxPP} converge to the Monge-Amp\`ere measure of a solution to \eqref{CxMAEq}. However, it should be stressed that when $\beta<0$ there is no guarantee that \eqref{CxPP} has finite volume and can be normalized to a probability measure. This turns out to be a subtle property and in one of the most famous versions of equation \eqref{CxMAEq}, when $M$ is a Fano manifold and $\omega_0$, $\mu$ and $\beta$ are chosen so that solutions to \eqref{CxMAEq} define K\"ahler-Einstein metrics of positive curvature, this reduces to a property of the manifold $M$ which is conjectured to be equivalent to the existence of K\"ahler-Einstein metrics on $M$ (see \cite{Kento} for some progress on this). We will explain in Section~\ref{SectionEqPush} how equation \eqref{MAEqSpec} can be seen as the ''push forward'' to a real setting of a complex Monge-Amp\`ere equation whose solution define K\"ahler-Einstein metrics of almost everywhere positive curvature. In that sense, the present project can be seen as an attempt to study one side of this complex geometric problem.

%%%%%%%%%%%%%%%%%%%%

\subsection{Equation \eqref{MAEqSpec} as the "Push Forward" of a Complex Monge-Amp\`ere Equation}\label{SectionEqPush}
Let $M=\C^n/(4\pi\Z^n+i\Z^n)$ and $\theta$ be the function on $\C^n$ defined as
$$ \theta(z)=\sum_{m\in \Z^n} e^{-m^2/4+izm/2}. $$
This is the classical $\theta$-function and it satisfies the following transformation properties:
\begin{eqnarray} 
\theta(z+4\pi) & = & \theta(z) \nonumber \\
\theta(z+i) & = & \theta(z)e^{iz/2-1/4}. \nonumber 
\end{eqnarray}
In particular, the zero set of $\theta$ defines the \emph{theta divisor}, $D$, on $M$ and, using certain trivializations of the line bundle associated to $D$, $\theta$ descends to a holomorphic section of this line bundle. This means $\tau = i\partial\bar\partial \log |\theta|^2$ is a well-defined (1,1)-current on $M$ and we may consider the twisted K\"ahler-Einstein equation
\begin{equation} \Ric (\omega) + \tau = \omega \label{EqTKEPos} \end{equation}
on $M$, where $\Ric(\omega)$ denotes the Ricci curvature of $\omega$. The current $\tau$ is supported on $D$ so away from $D$ this equation define metrics of constant positive Ricci curvature. Now, there is a standard procedure to rewrite \eqref{EqTKEPos} into a scalar equation of type \eqref{CxMAEq}. This process involves choosing a reference form $\omega_0$ in the cohomology class of $\tau$ and fixing a Ricci-potential of $\omega_0$, $F$, such that 
$$i\partial \bar\partial F = \Ric(\omega_0)+\tau-\omega_0. $$ 
Choosing $\omega_0=\sum_i idz_i\wedge d\bar z_i$ and $F=-y^2/2+\log |\theta|^2$ gives the equation 
\begin{equation} \MA_\C(u) = e^{-u-y^2/2}|\theta|^2\omega_0^n. \label{CxMAEqSpec} \end{equation}
In other words, we arrive at equation \eqref{CxMAEq} with the choices 
$$\mu_\C = |\theta^2|e^{-y^2/2} \omega_0^n$$ 
and 
$\beta=-1$. Now, let $z=x+iy$ be the standard coordinates on $M$ induced from $\C^n$. Let $\rho:M\rightarrow X$ be the map $z\mapsto y$. If $\phi$ is a twice differentiable function on $X$ such that $(\phi_{ij}+\delta_{ij})$ is strictly positive definite, then $u(z) := \phi(y)$ defines a (rotationally invariant) twice differentiable function on $M$ satisfying $i\partial\bar\partial u +\omega_0>0$. Moreover,   
\begin{equation} \rho_* \MA_\C(u) = \MA(\phi) \label{CxReMA} \end{equation}
where $\MA(u)$ is the complex Monge-Amp\`ere measure on $M$ defined in \eqref{CxMAOp} and $\MA(\phi)$ is the real Monge-Amp\`ere measure on $X$ defined in \eqref{MAOp}. Further, at the end of the next sextion we will prove
\begin{lemma}\label{LemmaThetaPush}
\begin{equation} \rho_* \left(e^{-y^2/2}|\theta|^2\omega_0^n\right) = \gamma \label{EqThetaPush} \end{equation}
where $dy$ is the uniform measure on $X$.
\end{lemma}
Since $u$ is rotationally invariant we get that
$$ \rho_* \left(\MA_\C(u)-e^{-u-y^2/2}|\theta|^2\omega_0^n\right) = \MA(\phi)-e^{-\phi}\gamma $$
and this is the relation that makes us refer to equation \eqref{MAEqSpec} as the ''push forward'' of equation \eqref{CxMAEqSpec}. 

%%%%%%%%%%%%%%%%%

\subsection{Permanental Point Processes as the Push Forward of Determinantal Point Processes}\label{SectionDetPerm}
Here we will establish a connection between the permanental point processes defined in Section \ref{SectPointProc} and the determinantal point processes defined in Bermans framework. The connection is a consequence of a certain formula that relates integrals of determinants to permanents. This formula might be of independent interest and is given in the following lemma.  
\begin{lemma}\label{LemmaDetPerm}
Let $(E,\mu)$ be a measure space. Let $N\in \N$ and 
$$\{F_{jk}: j=1\ldots N, k=1\ldots N \}$$ 
be a collection of complex valued functions on $E$, square integrable with respect to $\mu$, such that, for each $j$
$$ \int_E F_{jk}\overline{F_{jl}} d\mu = 0 $$
if $k\not=l$. Then 
$$ \perm\left(\int_E |F_{jk}|^2d\mu  \right) = \int_{E^N}|\det(F_{jk}(x_j))|^2d\mu^{\otimes N}. $$
\end{lemma}
\begin{proof}
Now, 
\begin{eqnarray} 
&  & \int_{E^N}|\det(F_{jk}(x_j))|^2 d\mu^{\otimes N} \nonumber \\
& = & \int_{E^N}\det(F_{jk}(x_j))\overline{\det(F_{jk}(x_j))}d\mu^{\otimes N}  \nonumber \\
& = & \int_{E^N} \left(\sum_\sigma (-1)^\sigma\prod_j F_{j\sigma(k)}(x_j)\right)\overline{\left(\sum_{\sigma'} (-1)^{\sigma'}\prod_j F_{j\sigma'(k)}(x_j)\right)}d\mu^{\otimes N}  \nonumber \\
& = & \sum_{\sigma,\sigma'} (-1)^{\sigma+\sigma'} \prod_j \int_E F_{j\sigma(k)}\overline{F_{j\sigma'(k)}} d\mu  \label{IntDet} 
\end{eqnarray}
By the orthogonality assumption on $\{F_{jk}\}_k$, the only contribution comes from terms where $\sigma=\sigma'$. We get
$$ \eqref{IntDet} = \sum_\sigma \prod_j \int_E |F_{j\sigma(k)}|^2 d\mu = \perm\left(\int_E |F_{jk}|^2d\mu\right). \qedhere $$
\end{proof}
Before we examine its consequences for permanental point processes we illustrate two other applications. The first is given by a quick proof of the following well known formula related to Gram Determinants (see for example \cite{Deift}):
\begin{corollary}
Let $(E,\mu)$ be a measure space and 
$$f_1,\ldots,f_N\in L^2(\mu).$$
Then
\begin{equation} \det\left(\int_Ef_j\overline{f_k}d\mu\right) = \frac{1}{N!}\int_{E^N} \left|\det \left(f_k(x_j)\right)\right|^2 d\mu^{\otimes N}. \label{DeiftFormula} \end{equation}
\end{corollary}
\begin{proof}
Note that if $A$ is an invertible $N\times N$ matrix with determinant $1$, then replacing $\{f_1,\ldots, f_n\}$ by $\{\tilde f_1,\ldots,\tilde f_N\}$ where $\tilde f_i$ is defined by 
$$(\tilde f_1,\ldots,\tilde f_N) = (f_1,\ldots,f_N)A$$ 
doesn't affect the formula \eqref{DeiftFormula}. This means we may assume $f_1,\ldots,f_N$ satisfy
$$ \int_E f_j\overline{f_k}d\mu = 0 $$
if $j\not=k$. 
For each $j,k\in \{1,\ldots,N\}$, let $F_{jk}=f_k$. We get that 
$$ \det\left(\int_Ef_j \overline{f_k} d\mu\right) = \prod_k \int_E|f_k|^2d\mu = \frac{1}{N!}\perm\left(\int_E|F_{jk}|^2 d\mu\right) $$
and, applying Lemma~\ref{LemmaDetPerm}, that 
\begin{eqnarray} 
\det\left(\int_Ef_j \overline{f_k} d\mu\right) & = & \frac{1}{N!}\perm\left(\int_E|F_{jk}|^2 d\mu\right) = \frac{1}{N!}\int_{E^N}\left|\det\left(F_{jk}(x_j)\right)\right|^2d\mu^{\otimes N} \nonumber \\
& = & \frac{1}{N!}\int_{E^N} \left|\det \left(f_k(x_j)\right)\right|^2 d\mu^{\otimes N} \nonumber
\end{eqnarray}
proving the corollary.
\end{proof}

The second application of Lemma~\ref{LemmaDetPerm} is given by the following formula for the permanent of a matrix of non-negative real numbers.
\begin{corollary}
Let $(a_{jk})$ be an $N\times N$-matrix of non-negative real numbers. Then
$$ \perm(a_{jk}) = \frac{1}{(2\pi)^N}\int_{[0,2\pi]^N} \left|\det \left(\sqrt{a_{jk}}e^{ikx_j}\right)\right|^2 dx_1\cdots dx_N. $$
\end{corollary}
\begin{proof}
Let $F_{jk}=\sqrt{a_{jk}}e^{ikx}.$ 
Then, for each $j$,
$$ \int_{[0,2\pi]} F_{jk}\overline{F_{jl}} dx = \int_{[0,2\pi]}a_{jk}e^{i(k-l)x}dx = \begin{cases} 2\pi a_{jk} & \textnormal{ if } l=k \\ 0 & \textnormal{ otherwise.} \end{cases} $$
Applying Lemma~\ref{LemmaDetPerm} gives
\begin{eqnarray} \perm(a_{jk}) & = & \frac{1}{(2\pi)^N}\perm \int_{[0,2\pi]} |F_{jk}|^2 dx \nonumber \\ 
& = & \frac{1}{(2\pi)^N}\int_{[0,2\pi]^N} \left|\det\left(F_{jk}(x_j)\right)\right|^2 dx_1\ldots dx_N \nonumber \\
& = & \frac{1}{(2\pi)^N}\int_{[0,2\pi]^N} \left|\det \left(\sqrt{a_{jk}}e^{ikx_j}\right)\right|^2 dx_1\cdots dx_N. \nonumber
\end{eqnarray}
which proves the corollary.
\end{proof}

To see how Lemma~\ref{LemmaDetPerm} connects permanental point processes to determinantal point processes, we will now look a bit closer on the point processes defined by Bermans framework when applied to the complex Monge-Amp\`ere equation in Section~\ref{SectionEqPush}. First of all, $\omega_0=\sum_i idz_i\wedge d\bar z_i $ represents the curvature class of the theta divisor $D$ on $M$. Elements in $H^0(M,kD)$ may be represented by theta functions and a basis at level $k\in \N$ is given by the set
\begin{equation} \{\theta^{(k)}_p: p\in \frac{1}{k}\Z^n/\Z^n\} \label{EqThetaColl} \end{equation}
where
$$ \theta^{(k)}_p = \sum_{m\in \Z^n+p}e^{-km^2/4+izkm/2}. $$
With respect to these trivializations the norm of $\theta_p^{(k)}$ with respect to the metric on $kD$ with curvature form $k\omega_0$ may be written
$$ \norm{\theta^{(k)}_p}^2 = |\theta^{(k)}_p|^2 e^{-ky^2/2}. $$
Enumeration the points in $\frac{1}{k}\Z^n/\Z^n$, $\{p_1,\ldots, p_N\}$ and using the standard coordinates $(z_1,\ldots, z_N)=(x_1+iy_1,\ldots,x_N+iy_N)$ on $M^N$ allow us to write the determinant in \eqref{CxPP} as 
$$ \left|\det\left(\theta^{(k)}_{p_l}(z_j)e^{-y_j^2/4}\right)_{jl}\right|^2. $$

Now, recall that the real Monge-Amp\`ere measure on $X$ may be recovered as the push forward under the projection map, $\rho:M\rightarrow X$, of the complex Monge-Amp\`ere measure on $M$ (see equation \eqref{CxReMA}). Similarly, Lemma~\ref{LemmaDetPerm} will allow us to explicitly calculate the push forward of the measure
$$ |\det(\theta^{(k)}_{p_i}(z_j)e^{-y_j^2/4})|^2 \omega_0^n $$
on $M^N$ under the map $\rho^{\times N}:M^N \rightarrow X^N$. We get the following lemma, which is the key point of this section. It shows that the permanental point processes defined in Section~\ref{SectPointProc} are the natural analog of the determinantal point processes defined by Bermans framework for complex Monge-Amp\`ere equations.
\begin{lemma}\label{LemmaDetPush}
Let $dy$ be the uniform measure on $X$. Then
\begin{equation} \left(\rho^{\times N}\right)_* |\det(\theta^{(k)}_{p_l}(z_j)e^{-y_j^2/4})|^2 \omega_0^n = \perm \left(\Psi^{(N)}_{p_l}(y_j)\right)dy. \label{EqDetPush} \end{equation}
\end{lemma}
\begin{proof}
Let $y=(y_1,\ldots,y_N)\in X^N$. The point $y\in X^N$ defines a real torus, $T_y$, in $M^N$
$$ T_y=\left(\rho^{\times N}\right)^{-1}(y) = \left\{x+iy: x\in (\R^n /4\pi\Z^n)^N\right\}. $$
If we let $dx$ be the measure on $T_y$ induced by $(\R^n)^N$, then the density at $y$ of the left hand side of \eqref{EqDetPush} with respect to $dy$ is given by the integral
\begin{equation} \int_{T_y} |\det(\theta^{(k)}_{p_l}(z_j))e^{-y_j^2/4}|^2 dx. \label{FibInt} \end{equation}
For each $j,l\in\{1,\ldots,N\}$, let $F_{jl}:T_y\rightarrow \C$ be defined by 
\begin{eqnarray} F_{jl}(x) & = & \theta_{p_l}(x+iy_j)e^{-y_j^2/4} \nonumber \\
& = & \sum_{m\in \Z^n+p_l}e^{-km^2/4+i(x+iy_j)km/2-y_j^2/4} \nonumber \\
& = & \sum_{m\in \Z^n+p_l}e^{-k(m-y_j)^2/4+ikmx/2} \nonumber
\end{eqnarray}
Now, when computing the integral
\begin{eqnarray} & & \int_{T_y}  F_{jl}\overline{F_{jl'}} dx \nonumber \\ 
& = & \int_{T_y} \sum_{\substack{m \in \Z^n+p_l \\ m'\in \Z^n+p_{l'}}} e^{-k(m-y_j)^2/4 -k(m'-y_j)^2/4 + ik(m-m')x/2} dx \label{IntTheta}
\end{eqnarray}
the only contribution comes from the terms where $m-m'=0$. If $l\not= l'$, then there are no such terms, in other words $\eqref{IntTheta}=0$. If $l=l'$ we are left with
$$\eqref{IntTheta} =  (4\pi)^N \sum_{m\in \Z+p_l} e^{-k|y_j-m|^2/2} = (4\pi)^N\Psi^{(N)}_{p_l}(y_j). $$
Applying Lemma~\ref{LemmaDetPerm} gives
$$ \eqref{FibInt} = \int_{T_y} \left|\det\left(F_{jl}(x_j)\right)\right|^2 dx = \perm \left(\int |F_{jl}|^2 dx\right) = \perm \left(\Psi_{p_l}(y_j)\right) $$
proving the lemma. 
\end{proof}
Finally, we show that Lemma~\ref{LemmaThetaPush} is a special case of this. 
\begin{proof}[Proof of Lemma~\ref{LemmaThetaPush}]
Note that $\theta = \theta_0^{(1)}$ and 
$$\gamma = \sum_{m\in \Z^n} e^{-|y-m]^2/2} dy = \Psi_0^{(1)} dy.$$
This means \eqref{EqThetaPush} is the special case of \eqref{EqDetPush} given by $N=k=1$. Hence the lemma follows from Lemma~\ref{LemmaDetPush}.
\end{proof}

\subsection{Approximations of Optimal Transport Maps}\label{SectApprox}
As mentioned in the introduction the point processes defined here can be used to produce explicit approximations of optimal transport maps. In optimal transport it is natural to consider a larger class of Monge-Amp\`ere operators. Let $\nu_0\in\M_1(X)$ be absolutely continuous with respect to $dx$. Then $\nu_0$ defines a Monge-Amp\`ere operator $\MA_{\nu_0}$ on $P(X)$ as
$$ \MA_{\nu_0}(\phi) = (\nabla^c \phi^c)_* \nu_0. $$
Solutions, $\phi_*$, to the inhomogenous Monge-Amp\`ere equation
\begin{equation} \MA_{\nu_0}(\phi)=\mu_0 \label{MAEqIH} \end{equation}
determine optimal transport maps  on $X$ in the sense that $T=\nabla^c \phi_*$ is the optimal \emph{transport map} in the sense of Brenier (see \cite{Villani}) from the source measure $\mu_0$ to the target measure $\nu_0$. 

The fact that the point processes defined in Section \ref{SectPointProc} are related to the standard $\MA=\MA_{dx}$ is a consequence of the fact that
$$ \frac{1}{N}\sum_{p\in \frac{1}{k}\Z^n/\Z^n} \delta_{p}\rightarrow dx $$
in the weak*-topology. Redefining $S^{(N)}$ in the following way will provide the generalisation we want: Let $P^{(N)}$ be a collection of point sets with the property that $|P^{(N)}|=N$ and 
$$ \frac{1}{N}\sum_{p\in P^{(N)}} \delta_{p}\rightarrow \nu_0. $$
As in the original definition, associate a wave function, $\Psi^{(N)}_p$, to each point $p\in \cup P^{(N)}$
$$ \Psi^{(K)}_{p_i} = \sum_{m\in \Z^n+p_i} e^{-|x-m|^2} $$
and, for each $N$, enumerate the points in $P^{(N)}$
$$ P^{(N)} = \{p_1,\ldots, p_N\}. $$
We get
\begin{corollary}\label{CorrOptTrans}
Let $\mu_0, \nu_0\in\M_1(X)$ be absolutely continuous and have smooth, strictly positive densities with respect to $dx$ and $\Psi^{(N)}_{p_i}$ be defined as above. Then 
$$ \phi_N := \frac{1}{N}\log \int_{X^{N-1}}\perm\left(\Psi^{(N)}_{p_i}(x_j)\right) d\mu^{\otimes N} $$
converges uniformly to the unique, smooth, strictly convex solution of \eqref{MAEqIH}. Consequently, the associated gradient maps $\nabla^c \phi_N$ converges uniformly to the unique optimal transport map transporting  $\mu_0$ to $\nu_0$. 
\end{corollary}
\begin{proof}[Proof of Corollary \ref{CorrOptTrans}]
First of all, the fact that the optimal transport map is smooth follow from Caffarelli's regularity theory for Monge-Amp\`ere equations. We will not go through the argument as it is similar as in Section \ref{SectRegularity}. Uniqueness is a basic result from optimal transport (see for example Theorem 2.4.7 in \cite{Villani}). Now, to see that the convergence holds, consider the functionals, $\{H^{(N)}\}$, on $\M_1(X)$ defined by
$$ E^{(N)}(\mu) = \frac{1}{N}\int_{X^N} H^{(N)} d\mu^{\otimes N}. $$
Direct calculations give that they are continuous, convex, Gateaux differentiable and $dE^{(N)}|_{\mu_0} = \Phi^{(N)}$. We claim that 
\begin{equation} E^{(N)}(\mu)\rightarrow W^2(\mu,dx) \label{EqCorr1} \end{equation}
for all $\mu\in \M_1(X)$. To see this, note that by the proof of Theorem \ref{ThmRedToZeroTemp}
$$ \sup_{X^N}\left|\frac{1}{N}H^{(N)}-W^2(\cdot,dx)\circ \delta^{(N)} \right| \rightarrow 0 $$
as $N\rightarrow \infty$. We get, since $\{\frac{1}{N}H^{(N)}\}$ are uniformly bounded,
\begin{eqnarray} E^{(N)}(\mu) & = & \int_{X^N} W^2(\cdot,dx)\circ \delta^{(N)} d\mu^{\otimes N} +o(1) \nonumber \\
& = & \int_{\M_1(X)} W^2(\cdot,dx) \left(\delta^{(N)}\right)_* \mu^{\otimes N} +o(1).\label{EqCorr3} 
\end{eqnarray}
where $o(1)\rightarrow 0$ as $N \rightarrow \infty$. Now, it follows from Sanov's theorem that $(\delta^{(N)})_*\mu^{\otimes N}\rightarrow \delta_{\mu}$ in the weak*-topology on $\M_1(\M_1(X))$. Now, since $X$ has finite diameter we get that the squared distance function on $X$ can be bounded by a a constant times the distance function. As the Wasserstein 1-metric metricizes the weak* topology on $\M_1(X)$ this implies that $W^2(\cdot,dx)$ is continuous on $\M_1(X)$. We get that \eqref{EqCorr3} converges to $W^2(\mu,dx)$ as $N\rightarrow \infty$.

Further, $W^2(\cdot,dx)$ is convex. By standard properties of convex functions $dE^{(N)}|_{\mu_0}$ converges to a subgradient of $W^2(\cdot,dx)$ at $\mu_0$. By standard properties of the Legendre Transform this means
\begin{equation} \phi = \lim_{N\rightarrow \infty} \phi^{(N)} \label{EqCorr2} \end{equation}
satisfies $d\xi|_\phi = MA(\phi) = \mu_0$. This means $\phi$ is smooth and $\nabla^c \phi$ defines the optimal transport map transporting $\mu_0$ to $\nu_0$. Now, let $\Phi_N$ and $\Phi$ be the images in $P_{\Z^n}(\R^n)$ of $\phi_N$ and $\phi$ respectively. The convergence in \eqref{EqCorr2} implies $\Phi_N\rightarrow \Phi$ and, by standard properties of convex functions, $\nabla\Phi^{(N)}\rightarrow \nabla \Phi$. This means $\nabla^c\phi_N \rightarrow \nabla^c\phi$ which proves the Corollary.
\end{proof}

\end{document}